\documentclass[a4paper]{amsart}

\usepackage{amssymb}
\usepackage[all]{xy}
\usepackage{hyperref,aliascnt}
\usepackage{mathtools,array}

\numberwithin{equation}{section}

\newtheorem{lma}{Lemma}[section]

\newaliascnt{thmCt}{lma}
\newtheorem{thm}[thmCt]{Theorem}
\aliascntresetthe{thmCt}

\newaliascnt{corCt}{lma}
\newtheorem{cor}[corCt]{Corollary}
\aliascntresetthe{corCt}

\newaliascnt{prpCt}{lma}
\newtheorem{prp}[prpCt]{Proposition}
\aliascntresetthe{prpCt}

\newtheorem*{thm*}{Theorem}
\newtheorem*{cor*}{Corollary}
\newtheorem*{prp*}{Proposition}

\theoremstyle{definition}

\newaliascnt{pgrCt}{lma}
\newtheorem{pgr}[pgrCt]{}
\aliascntresetthe{pgrCt}

\newaliascnt{dfnCt}{lma}
\newtheorem{dfn}[dfnCt]{Definition}
\aliascntresetthe{dfnCt}

\newaliascnt{rmkCt}{lma}
\newtheorem{rmk}[rmkCt]{Remark}
\aliascntresetthe{rmkCt}

\newaliascnt{rmksCt}{lma}

\aliascntresetthe{rmksCt}

\newaliascnt{qstCt}{lma}
\newtheorem{qst}[qstCt]{Question}
\aliascntresetthe{qstCt}

\newaliascnt{pbmCt}{lma}
\newtheorem{pbm}[pbmCt]{Problem}
\aliascntresetthe{pbmCt}

\newaliascnt{exaCt}{lma}
\newtheorem{exa}[exaCt]{Example}
\aliascntresetthe{exaCt}

\newaliascnt{ntnCt}{lma}
\newtheorem{ntn}[ntnCt]{Notation}
\aliascntresetthe{ntnCt}

\newcommand{\NN}{\mathbb{N}}
\newcommand{\QQ}{\mathbb{Q}}

\newcommand{\CC}{\mathbb{C}}

\newcommand{\pom}{positively ordered monoid}

\newcommand{\ca}{$C^*$-al\-ge\-bra}

\newcommand{\axiomO}[1]{(O#1)}

\DeclareMathOperator{\Lat}{Lat}

\DeclareMathOperator{\Cu}{Cu}

\DeclareMathOperator{\ev}{ev}

\DeclareMathOperator{\id}{id}

\newcommand{\freeVar}{\_\,}

\usepackage{stmaryrd}
\newcommand{\ihom}[1]{\llbracket #1 \rrbracket}

\newcommand{\CatCu}{\ensuremath{\mathrm{Cu}}}
\newcommand{\CatCuMor}{\CatCu}
\newcommand{\CatCuBimor}{\mathrm{Bi}\CatCu}

\newcommand{\CatQ}{\mathcal{Q}}

\newcommand{\NNbar}{\overline{\mathbb{N}}}
\newcommand{\PPbar}{\overline{\mathbb{P}}}
\DeclareMathOperator{\ltIso}{l}
\DeclareMathOperator{\rtIso}{r}
\DeclareMathOperator{\iIso}{i}
\DeclareMathOperator{\unit}{d}
\DeclareMathOperator{\counit}{e}

\newcommand{\tensExt}{\boxtimes}
\newcommand{\pathCu}[1]{\mathbf{#1}}
\newcommand{\andSep}{\,\,\,\text{ and }\,\,\,}

\newcommand{\CuSgp}{$\CatCu$-sem\-i\-group}
\newcommand{\CuSrg}{$\CatCu$-sem\-i\-ring}
\newcommand{\CuMor}{$\CatCu$-mor\-phism}

\begin{document}

\title{Abstract bivariant Cuntz semigroups II}
\author{Ramon Antoine}
\author{Francesc Perera}
\author{Hannes Thiel}

\date{\today}

\address{
Ramon~Antoine, Departament de Matem\`{a}tiques,
Universitat Aut\`{o}noma de Barcelona,
08193 Bellaterra, Barcelona, Spain}
\email[]{ramon@mat.uab.cat}

\address{
Francesc~Perera, Departament de Matem\`{a}tiques,
Universitat Aut\`{o}noma de Bar\-ce\-lo\-na,
08193 Bellaterra, Barcelona, Spain}
\email[]{perera@mat.uab.cat}

\address{
Hannes~Thiel,
Mathematisches Institut,
Universit\"at M\"unster,
Einsteinstrasse 62, 48149 M\"unster, Germany}
\email[]{hannes.thiel@uni-muenster.de}




\begin{abstract}
We previously showed that abstract Cuntz semigroups form a closed symmetric monoidal category.
This automatically provides additional structure in the category, such as a composition and an external tensor product, for which we give concrete constructions in order to be used in applications.

We further analyse the structure of not necessarily commutative $\CatCu$-semi\-rings and we obtain, under mild conditions, a new characterization of solid $\CatCu$-semirings $R$ by the condition that $R\cong\ihom{R,R}$.
\end{abstract}

\maketitle

\section{Introduction}
\label{sec:intro}

The Cuntz semigroup is a geometric refinement of $K$-theory that was introduced by Cuntz \cite{Cun78DimFct} in his groundbreaking studies of simple \ca{s}.
It is a partially ordered semigroup that is constructed from positive elements in the stabilization of the algebra, in a similar way to how the Murray-von Neumann semigroup in $K$-theory is built from projections.

There has been an intensive use of Cuntz semigroups in \ca{} theory, particularly related to classification results.
One of the most prominent instances of this appeared in Toms' example \cite{Tom08ClassificationNuclear} of two simple \ca{s} that are indistinguishable by $K$-theoretic and tracial data, yet their Cuntz semigroups are not order-isomorphic.
For the largest agreeable class of simple, separable, nuclear, finite and $\mathcal{Z}$-stable \ca{s} that can be classified using $K$-theoretic and tracial information, the Cuntz semigroup features prominently since, suitably paired with $K_1$, it is functorially equivalent to the Elliott invariant;
see \cite{BroPerTom08CuElliottConj} and \cite{AntDadPerSan14RecoverElliott}.

Further, for \ca{s} of stable rank one, the Cuntz semigroup has additional fine structure that was key in solving a number of open problems, such as the Blackadar Handelman conjecture and the Global Glimm Halving Problem;
see \cite{AntPerRobThi18arX:CuntzSR1}. 

The formal framework to study Cuntz semigroups of \ca{s} is provided by the category $\CatCu$ of abstract Cuntz semigroups, which was introduced by Coward, Elliott and Ivanescu in \cite{CowEllIva08CuInv}, and which was studied in detail in \cite{AntPerThi18:TensorProdCu} and later in \cite{AntPerThi17arX:AbsBivariantCu}.
The objects of this category, called \CuSgp{s}, are partially ordered semigroups that satisfy the order-theoretic analogues of being a complete topological space without isolated points. 
The $\CatCu$-morphisms are natural models for $^*$-homomorphisms between \ca{s}. 
The following result was established in \cite{AntPerThi18:TensorProdCu} and \cite{AntPerThi17arX:AbsBivariantCu}.

\begin{thm*}
The category $\CatCu$ of abstract Cuntz semigroups is a symmetric monoidal closed category. 
\end{thm*}

This means that there are bifunctors
\[
\freeVar\otimes\freeVar\colon\CatCu\times\CatCu\to\Cu, \andSep
\ihom{\freeVar,\freeVar}\colon\CatCu\times\CatCu\to\CatCu,
\]
such that $\otimes$ is associative, symmetric, has as unit object the semigroup $\NNbar=\{0,1,\dots,\infty\}$ and, for any $\CatCu$-semigroup $T$, the functors $\freeVar\otimes T$ and $\ihom{T,\freeVar}$ are an adjoint pair.
For $\CatCu$-semigroups $S$ and $T$, the construction of the internal-hom $\ihom{S,T}$ is based on paths of the so-called generalized $\CatCu$-morphisms, which are natural models for completely positive contractive order-zero maps between \ca{s}; 
see \autoref{sec:prelims} for more details. We refer to the semigroups $\ihom{S,T}$ as bivariant Cu-semigroups. 

Further, we show in \cite{AntPerThi18pre:Productscoproducs} that $\CatCu$ is complete and cocomplete, and that the functor that assigns to each \ca\ its Cuntz semigroup is compatible with products and ultraproducts. 

The fact that $\CatCu$ is a closed category automatically adds additional features well known to category theory (see, for instance, \cite{Kel05EnrichedCat}).
For example, one obtains a \emph{composition product} given in the form of a $\CatCu$-morphism:
\[
\circ\colon\ihom{T,P}\otimes\ihom{S,T}\to\ihom{S,P},
\]
which is the generalization of the composition of morphisms in a category to a notion of composition between internal-hom objects in a closed category; see the comments after \cite[Proposition~5.11]{AntPerThi17arX:AbsBivariantCu}.

Although the said features can be derived from general principles, in our setting they become concrete, and this is very useful in applications.
In this direction, and bearing in mind that $\ihom{S,T}$ is a semigroup built out of paths of morphisms from $S$ to $T$, the composition product can be realized as the composition of paths.
Another important example is the \emph{evaluation map} which, for $\CatCu$-semigroups $S$ and $T$ is a $\CatCu$-morphism $\counit_{S,T}\colon\ihom{S,T}\otimes S\to T$ such that $\counit_{S,T}(x\otimes a)$ can be interpreted as the evaluation of $x\in\ihom{S,T}$ at $a\in S$.
We therefore also write $x(a):=\counit_{S,T}(x\otimes a)$.
The evaluation map can be used to concretize the adjunction between the internal-hom bifunctor and the tensor product;
see \autoref{prp:genProp:correspondence}.

Likewise, the tensor product of generalized $\CatCu$-morphisms induces an \emph{external tensor product}
\[
\tensExt\colon \ihom{S_1,T_1}\otimes\ihom{S_2,T_2}\to\ihom{S_1\otimes S_2, T_1\otimes T_2},
\]
which is associative and, like in $KK$-Theory, compatible with the composition product;
see \autoref{prp:genProp:tensExtComp}.
This means that, for elements $x_k\in\ihom{S_k,T_k}$ and $y_k\in \ihom{T_k,P_k}$ (for $k=1,2$), we have
\[
(y_2\tensExt y_1)\circ (x_2\tensExt x_1)=(y_2\circ x_2)\tensExt (y_1\circ x_1).
\]

In \autoref{sec:bivarCu:functoriality}, we study the ideal structure of bivariant \CuSgp{s}.
Given an ideal $J$ in $S$, and an ideal $K$ in $T$, we show that there is a natural identification of $\ihom{S/J,K}$ with an ideal in $\ihom{S,T}$;
see Propositions~\ref{prp:bivarCu:idealSecondEntry} and~\ref{prp:bivarCu:idealFirstEntry}.
However, in general, not every ideal of $\ihom{S,T}$ arises this way.
Indeed, in \autoref{exa:bivarCu:ihomSSNotSimple} we construct a simple \CuSgp{} $S$ such that $\ihom{S,S}$ is not simple.

In \autoref{sec:ring}, we deepen our study of $\CatCu$-semirings and their semimodules, which was started in \cite[Chapters~7 and~8]{AntPerThi18:TensorProdCu}. 
Given a \CuSgp{} $S$, the composition product turns $\ihom{S,S}$ into a $\CatCu$-semiring;
see \autoref{prp:ring:ihomSS}.
Further, the evaluation map defines natural left $\ihom{S,S}$-action on $S$;
see \autoref{prp:ring:ihomSS_act_S}.
Finally, $\ihom{S,T}$ has both a natural left $\ihom{T,T}$-action and a compatible right $\ihom{S,S}$-action;
see \autoref{prp:ring:ihomST}.

For any $\CatCu$-semiring $R$, the internal-hom construction makes it possible to define a left regular representation-like map $\pi_R\colon R\to\ihom{R,R}$, which is always a multiplicative order-embedding (and unital in case the unit of $R$ is a compact element);
see \autoref{dfn:ring:pi}, \autoref{prp:ring:piMultiplicative} and \autoref{prp:ring:R_sub_ihomRR}.
We study when $\pi_R\colon R\to\ihom{R,R}$ is an isomorphism;
see \autoref{prp:ring:charSolid}.
This is closely related to the property of being \emph{solid}, which means that the multiplication defines an isomorphism between $R\otimes R$ and $R$;
see \cite[Definition~7.1.6]{AntPerThi18:TensorProdCu}.



\section*{Acknowledgements}

This work was initiated during a research in pairs (RiP) stay at the Oberwolfach Research Institute for Mathematics (MFO) in March 2015.
The authors would like to thank the MFO for financial support and for providing inspiring working conditions.

Part of this research was conducted while the third named author was visiting the Universitat Aut\`{o}noma de Barcelona (UAB) in September 2015 and June 2016, and while the first and second named authors visited M\"unster Universit\"at in June 2015 and 2016. Part of the work was also completed while the second and third named authors were attending the Mittag-Leffler institute during the 2016 program on Classification of Operator Algebras: Complexity, Rigidity, and Dynamics.
They would like to thank
all the involved institutions for their kind hospitality.

The two first named authors were partially supported by MINECO (grants No.\ MTM2014-53644-P and No. MTM2017-83487-P), and by the Comissionat per Universitats i Recerca de la Generalitat de Catalunya.
The third named author was partially supported by the Deutsche Forschungsgemeinschaft (SFB 878 Groups, Geometry \& Actions).

\section{Preliminaries}
\label{sec:prelims}


Let $S$ be a positively ordered semigroup. Recall (cf. \cite[Definition~I-1.11, p.57]{GieHof+03Domains}) that an \emph{additive auxiliary relation} on $S$ is a binary relation $\prec$ on $S$ satisfying the following conditions for all $a,a',b,b'\in S$:
\begin{enumerate}
\item
If $a\prec b$ then $a\leq b$.
\item
If $a'\leq a\prec b\leq b'$ then $a'\prec b'$.
\item
We have $0\prec a$.
\item
The relation $\prec$ is compatible with addition.
\end{enumerate}

An important example of an auxiliary relation that we will use in the sequel is the \emph{way-below relation}, originally coming from Domain Theory (see \cite{GieHof+03Domains}):

Let $S$ be a positively ordered semigroup, and let $a,b\in S$.
Recall that $a$ is \emph{way-below} $b$ (we also say that $a$ is \emph{compactly contained in} $b$), denoted $a\ll b$, if whenever $(c_n)_n$ is an increasing sequence in $S$ for which the supremum exists and that satisfies $b\leq \sup_n c_n$, then there exists $k\in\NN$ with $a\leq c_k$.

We say that $a$ is \emph{compact} if $a\ll a$, and we let $S_c$ denote the submonoid of compact elements in $S$.

\begin{dfn}[{\cite{CowEllIva08CuInv}; see also \cite[Definition~3.1.2]{AntPerThi18:TensorProdCu}}]
\label{dfn:prelim:CatCu}
A \emph{$\CatCu$-semigroup}, also called \emph{abstract Cuntz semigroup}, is a positively ordered semigroup $S$ that satisfies the following axioms \axiomO{1}-\axiomO{4}:
\begin{itemize}
\item[\axiomO{1}]
Every increasing sequence $(a_n)_n$ in $S$ has a supremum $\sup_n a_n$ in $S$.
\item[\axiomO{2}]
For every element $a\in S$ there exists a sequence $(a_n)_n$ in $S$ with $a_n\ll a_{n+1}$ for all $n\in\NN$, and such that $a=\sup_n a_n$.
\item[\axiomO{3}]
If $a'\ll a$ and $b'\ll b$ for $a',b',a,b\in S$, then $a'+b'\ll a+b$.
\item[\axiomO{4}]
If $(a_n)_n$ and $(b_n)_n$ are increasing sequences in $S$, then $\sup_n(a_n+b_n)=\sup_n a_n+\sup_n b_n$.
\end{itemize}

A \emph{$\CatCu$-morphism} between \CuSgp{s} $S$ and $T$ is an additive map $f\colon S\to T$ that preserves order, the zero element, the way-below relation and suprema of increasing sequences. In case $f$ is not required to preserve the way-below relation, then we say it is a \emph{generalized $\CatCu$-morphism}. The set of $\CatCu$-morphisms (respectively, generalized  $\CatCu$-morphisms) is denoted by $\CatCuMor(S,T)$ (respectively, by $\CatCuMor[S,T]$).

We let $\CatCu$ be the category whose objects are $\CatCu$-semigroups and whose morphisms are $\CatCu$-morphisms.
\end{dfn}

A notion central to the construction of tensor products is that of bimorphisms, which we now recall.

\begin{dfn}[{\cite[Definition~6.3.1]{AntPerThi18:TensorProdCu}}]
\label{dfn:prelim:CatCuBimor}
Let $S,T$ and $P$ be $\CatCu$-semigroups.
A map $\varphi\colon S\times T\to P$ is a \emph{$\CatCu$-bimorphism} if it satisfies the following conditions:
\begin{enumerate}
\item $\varphi$ is a positively ordered monoid morphism in each variable.
\item
We have that $\sup_k\varphi(a_k,b_k)=\varphi(\sup_k a_k, \sup_k b_k)$, for every increasing sequences $(a_k)_k$ in $S$ and $(b_k)_k$ in $T$.
\item
If $a',a\in S$ and $b',b\in T$ satisfy $a'\ll a$ and $b'\ll b$, then $\varphi(a',b')\ll\varphi(a,b)$.
\end{enumerate}
\end{dfn}

The set of $\CatCu$-bimorphisms is denoted by $\CatCuBimor(S\times T,P)$. Equipped with pointwise order and addition, this set is a \pom. Similarly, the set of $\CatCu$-morphisms between two $\CatCu$-semigroups is also a \pom{}.

\begin{thm}[{\cite[Theorem~6.3.3]{AntPerThi18:TensorProdCu}}]
\label{prp:prelim:tensCu}
Let $S$ and $T$ be $\CatCu$-semigroups.
Then there exists a $\CatCu$-semigroup $S\otimes T$ and a $\CatCu$-bimorphism $\omega\colon S\times T\to S\otimes T$ such that for every $\CatCu$-semigroup $P$ the following universal properties hold:
\begin{enumerate}
\item
For every $\CatCu$-bimorphism $\varphi\colon S\times T\to P$ there exists a (unique) $\CatCu$-morphism $\tilde{\varphi}\colon S\otimes T\to P$ such that $\varphi=\tilde{\varphi}\circ\omega$.
\item
If $\alpha_1,\alpha_2\colon S\otimes T\to P$ are $\CatCu$-morphisms, then $\alpha_1\leq\alpha_2$ if and only if $\alpha_1\circ\omega\leq\alpha_2\circ\omega$.
\end{enumerate}
Thus, for every $\CatCu$-semigroup $P$, the assignment 
\[(\alpha\colon S\otimes T\to P)\mapsto (\alpha\circ\omega\colon S\times T\to P)\] defines a natural isomorphism of positively ordered monoids
\[
\CatCuMor\big( S\otimes T, P \big) \cong \CatCuBimor\big( S\times T, P \big).
\]
\end{thm}

The existence of a natural tensor product turns $\CatCu$ into a symmetric monoidal category;
see \cite[6.3.7]{AntPerThi18:TensorProdCu}.
As mentioned above, the tensor product functor $-\otimes T$ has a right adjoint $\ihom{T,-}$, and thus $\CatCu$ is also a closed category.
We recall some details;
see \cite[Section 3]{AntPerThi17arX:AbsBivariantCu} for a full account.

Let $(S,\prec)$ be an ordered semigroup equipped with an additive auxiliary relation $\prec$, and let $I_{\QQ}=\QQ\cap (0,1)$. A  \emph{path} on $S$ is a map $f\colon I_\QQ\to S$ such that $f(\lambda')\prec f(\lambda)$ whenever $\lambda'<\lambda$. The set of paths on $S$ is denoted by $P(S)$, which becomes a semigroup under pointwise addition. It is often the case that we write $f_\lambda=f(\lambda)$, and refer to a path as $\pathCu{f}=(f_\lambda)_{\lambda\in I_\QQ}$.

Given paths $f,g$ in $P(S)$, write $f\lesssim g$ if for every $\lambda\in I_\QQ$, there is $\mu\in I_\QQ$ such that $f(\lambda)\prec g(\mu)$. Set $f\sim g$ provided $f\lesssim g$ and $g\lesssim f$, and let $\tau(S,\prec):=P(S)/\!\sim$. Let $[f]$ denote the equivalence class of a path $f$. Then $\tau(S,\prec)$ becomes an ordered semigroup by setting $[f]+[g]=[f+g]$ and $[f]\leq [g]$ provided $f\lesssim g$.
It was proved in \cite[Theorem~3.15]{AntPerThi17arX:AbsBivariantCu} that, for $S$ as above, the semigroup $\tau(S)$ is a $\CatCu$-semigroup.

\begin{rmk}
\label{rmk:Qcategory}
The above construction is also referred to as the \emph{$\tau$-construction} in \cite{AntPerThi17arX:AbsBivariantCu}. It defines a functor $\tau\colon\CatQ\to\CatCu$, where $\CatQ$ is the category of positively ordered semigroups $S$ with an additive auxiliary relation $\prec$ that additionally satisfy axioms \axiomO{1} and \axiomO{4}.
In this way, $\tau$ is a coreflector of the inclusion functor $\iota\colon\CatCu\to\CatQ$; see \cite[Theorem~4.12]{AntPerThi17arX:AbsBivariantCu}.
Objects (respectively, morphisms) in the category $\CatQ$ are termed $\CatQ$-semigroups (respectively, $\CatQ$-morphisms).
\end{rmk}

If now $S$ and $T$ are $\CatCu$-semigroups, it is clear that $\CatCuMor[S,T]$ is also an ordered semigroup (with pointwise order and addition), and satisfies axioms \axiomO{1} and \axiomO{4} by taking pointwise suprema of increasing sequences.
Given $\varphi,\psi\in\CatCuMor[S,T]$, we define $\varphi\prec \psi$ provided $\varphi(a')\ll\psi(a)$ whenever $a'\ll a$.
This is easily seen to be an additive auxiliary relation on $\CatCu[S,T]$. Thus, $\CatCuMor[S,T]$ is a $\CatQ$-semigroup in the sense of \autoref{rmk:Qcategory}.

\begin{dfn}[{\cite[Definition 5.3]{AntPerThi17arX:AbsBivariantCu}}]
\label{dfn:absbiv}
Let $S$ and $T$ be \CuSgp{s}. The \emph{internal hom} from $S$ to $T$ is the \CuSgp{}
\[
\ihom{S,T} :=\tau\big( \CatCuMor[S,T],\prec \big).
\]
We also call $\ihom{S,T}$ the \emph{bivariant \CuSgp{}}, or the \emph{abstract bivariant Cuntz semigroup} of $S$ and $T$.
\end{dfn}

The proof that $\CatCu$ is a closed category requires the use of the so-called endpoint map, which is made precise below.

\begin{dfn}[{\cite[Definition 5.5]{AntPerThi17arX:AbsBivariantCu}}]
\label{dfn:endpoint}
Let $S$ and $T$ be \CuSgp{s}.
We let $\sigma_{S,T}\colon\ihom{S,T}\to\CatCuMor[S,T]$ be defined by
\[
\sigma_{S,T}([\pathCu{f}])(a)
= \sup_{\lambda\in I_\QQ} f_\lambda(a),
\]
for a path $\pathCu{f}=(f_\lambda)_\lambda$ in $\CatCu[S,T]$ and $a\in S$.
We refer to $\sigma_{S,T}$ as the \emph{endpoint map}.
\end{dfn}

\begin{thm}[{\cite[Theorem 5.9]{AntPerThi17arX:AbsBivariantCu}}]
\label{thm:Cuclosedbijection} 
Let $S, T$ and $P$ be \CuSgp{s}.
Then there are natural positively ordered monoid isomorphisms
\[
\CatCuMor\big( S, \ihom{T,P} \big)
\cong \CatCuBimor\big( S\times T, P \big)
\cong \CatCuMor\big( S\otimes T, P \big).
\]
The first isomorphism is given by 
\[
(\alpha\colon S\to\ihom{T,P})\mapsto (\tilde{\alpha}\colon S\times T\to P),
\] where $\tilde{\alpha}(a,b)=\sigma_{T,P}(\alpha(a))(b)$, for $(a,b)\in S\times T$.
The second is given by 
\[
(\beta\colon S\otimes T\to P)\mapsto ((a,b)\mapsto\beta(a\otimes b)),\text{ for }(a,b)\in S\times T.
\]
\end{thm}

\section{Concretization of categorical constructions for \texorpdfstring{$\CatCu$}{Cu}}
\label{sec:genProp}

In this section, we give concrete pictures of general constructions in closed, symmetric, monoidal categories for the category $\CatCu$. This will be used in the next section, and we start below with the analysis of the unit and counit maps.

\begin{dfn}
\label{dfn:genProp:unitMap}
Given \CuSgp{s} $S$ and $T$, the \emph{unit map} is the \CuMor{} $\unit_{S,T}\colon S\to\ihom{T,S\otimes T}$ that under the identification
\[
\CatCuMor\big( S, \ihom{T,S\otimes T} \big)
\cong \CatCuMor\big( S\otimes T, S\otimes T \big)
\]
corresponds to the identity map on $S\otimes T$.
\end{dfn}

In the result below we shall use that, if $S$ is a $\CatCu$-semigroup, and $a\in S$, then there is $(a_\lambda)_{\lambda\in I_\QQ}$ such that $a_{\lambda'}\ll a_\lambda$ whenever $\lambda'<\lambda$, $a_\lambda=\sup_{\lambda'<\lambda} a_{\lambda'}$, and $\sup_\lambda a_\lambda=s$;
see \cite[Proposition~2.8]{AntPerThi17arX:AbsBivariantCu}.

\begin{prp}
\label{prp:genProp:unitMap}
Let $S$ and $T$ be \CuSgp{s}, and let $a\in S$.
Let $(a_\lambda)_{\lambda\in I_\QQ}$ be a path in $(S,\ll)$ with endpoint $as$.
Then for each $\lambda\in I_\QQ$, the map $a_\lambda\otimes\freeVar\colon T\to S\otimes T$, sending $b\in T$ to $a_\lambda\otimes b$, is a generalized \CuMor{}.
Moreover, $(a_\lambda\otimes\freeVar)_{\lambda\in I_\QQ}$ is a path in $(\CatCuMor[T, S\otimes T],\prec)$, and we have $\unit_{S,T}(a) = [(a_\lambda\otimes\freeVar)_\lambda]$.
\end{prp}
\begin{proof}
The map $\omega\colon S\times T\to S\otimes T$, given by $\omega(s,t)=s\otimes t$, is a $\CatCu$-bimorphism.
This implies that $s\otimes\freeVar\colon T\to S\otimes T$ is a generalized \CuMor{} for each $s\in S$.
Moreover, using that $\omega$ preserves the joint way-below relation, we obtain that $s'\otimes\freeVar\prec s\otimes\freeVar$ for $s',s\in S$ satisfying $s'\ll s$.
In particular, if $(s_\lambda)_{\lambda\in I_\QQ}$ is a path in $S$, then $(s_\lambda\otimes\freeVar)_{\lambda\in I_\QQ}$ is a path in $(\CatCuMor[T, S\otimes T],\prec)$.
We define $\alpha\colon S\to\ihom{T,S\otimes T}$ by sending $s\in S$ to $[(s_\lambda\otimes\freeVar)_\lambda]$ for some choice of path $(s_\lambda)_\lambda$ in $S$ with endpoint $as$.
It is straightforward to check that $\alpha$ is a well-defined \CuMor{}.

Let us show that $\alpha=\unit_{S,T}$.
Consider the bijections
\[
\CatCuMor\big( S, \ihom{T,S\otimes T} \big)
\cong \CatCuBimor\big( S\times T, S\otimes T \big)
\cong \CatCuMor\big( S\otimes T, S\otimes T \big)
\]
from \autoref{thm:Cuclosedbijection}.
Under the first bijection, $\alpha$ corresponds to the $\CatCu$-bimorphism $\bar{\alpha}$ given by
\[
\bar{\alpha}(s,t)=\sigma_{T,S\otimes T}(\alpha(s))(t),
\]
for $s\in S$ and $t\in T$, where $\sigma_{T,S\otimes T}$ is the endpoint map.
We compute
\[
\bar{\alpha}(s,t)=\sigma_{T,S\otimes T}(\alpha(s))(t)
= \sup_{\lambda\in I_\QQ} (s_\lambda\otimes\freeVar)(t)
= \sup_{\lambda\in I_\QQ} (s_\lambda\otimes t)
= s\otimes t,
\]
for every path $(s_\lambda)_\lambda$ with endpoint $s\in S$, and every $t\in T$.
It follows that $\bar{\alpha}$ corresponds to $\id_{S\otimes T}$ under the second bijection.
By definition of $\unit_{S,T}$, this shows that $\alpha=\unit_{S,T}$, as desired.
\end{proof}

\begin{ntn}
\label{ntn:ihomisomorphism}
Given \CuSgp{s} $S$, $T$ and $P$, and a $\CatCu$-bimorphism $\alpha\colon S\times T\to P$, we shall often use the notation $\bar{\alpha}\colon S\to\ihom{T,P}$ to refer to the \CuMor{} that corresponds to $\alpha$ under the identification in \autoref{thm:Cuclosedbijection}.
\end{ntn}

\begin{cor}
\label{prp:genProp:unitMapEndpoint}
Let $S$ and $T$ be \CuSgp{s}.
Then the composition
\[
\sigma_{T,S\otimes T}\circ\unit_{S,T} \colon S \xrightarrow{\unit_{S,T}} \ihom{T,S\otimes T} \xrightarrow{\sigma_{T,S\otimes T}} \CatCuMor[T,S\otimes T].
\]
satisfies $(\sigma_{T,S\otimes T}\circ\unit_{S,T})(a)=a\otimes\freeVar$, for every $a\in S$.
In particular
\[
(\sigma_{T,S\otimes T}\circ\unit_{S,T})(a)(b) = a\otimes b,
\]
for $a\in S$ and $b\in T$.
\end{cor}
\begin{proof}
Let $a\in S$ and $b\in T$.
Choose a path $(a_\lambda)_\lambda$ in $S$ with endpoint $a$.
Then $\unit_{S,T}(a)=[(a_\lambda\otimes\freeVar)_\lambda]$ by \autoref{prp:genProp:unitMap}.
The supremum of the maps $a_\lambda\otimes\freeVar$ in $\CatCuMor[S,T\otimes S]$ is the map $a\otimes\freeVar$.
Thus, $(\sigma_{T,S\otimes T}\circ\unit_{S,T})(a)=a\otimes\freeVar$, as desired.
\end{proof}

\begin{dfn}
\label{dfn:genProp:counitMap}
Given \CuSgp{s} $S$ and $T$, the \emph{counit map} (also called \emph{evaluation map}) is the \CuMor{} $\counit_{S,T}\colon\ihom{S,T}\otimes S\to T$ that under the identification
\[
\CatCuMor\big( \ihom{S,T}\otimes S,T \big)
\cong \CatCuMor\big( \ihom{S,T},\ihom{S,T} \big)
\]
corresponds to the identity map on $\ihom{S,T}$.
Given $x\in\ihom{S,T}$ and $a\in S$, we also use $x(a)$ to denote $\counit_{S,T}(x\otimes a)$, but note that $x(a)=x'(a)$ for all $a\in S$ does not imply $x=x'$.
\end{dfn}

\begin{prp}
\label{prp:genProp:counitMap}
Let $S$ and $T$ be \CuSgp{s}, let $x\in\ihom{S,T}$, and let $a\in S$.
Then $\counit_{S,T}(x\otimes a)=\sigma_{S,T}(x)(a)$.
Thus, if $\pathCu{f}=(f_\lambda)_\lambda$ is a path in $\CatCuMor[S,T]$, then
\[
[\pathCu{f}](s)
= \counit_{S,T}([\pathCu{f}]\otimes a)
= \sup_{\lambda<1}f_\lambda(a).
\]
\end{prp}
\begin{proof}
Consider the bijections
\[
\CatCuMor\big( \ihom{S,T}, \ihom{S,T} \big)
\cong \CatCuBimor\big( \ihom{S,T}\times S, T \big)
\cong \CatCuMor\big( \ihom{S,T}\otimes S, T \big)
\]
from \autoref{thm:Cuclosedbijection}.
To simplify notation, we denote the identity map on $\ihom{S,T}$ by $\id$.
Under the first bijection, $\id$ corresponds to the $\CatCu$-bimorphism $\bar{\id}$ satisfying
\[
\bar{\id}(y,s)=\sigma_{S,T}(\id(y))(s),
\]
for all $y\in\ihom{S,T}$ and $s\in S$.
We obtain that
\[
\counit_{S,T}(x\otimes a)
= \bar{\id}(x,a)
=\sigma_{S,T}(\id(x))(a)
=\sigma_{S,T}(x)(a). \qedhere
\]
\end{proof}

\begin{rmk}
\label{rmk:genProp:counitMap}
Let $\varphi\colon S\to T$ be a \CuMor, and let $a\in S$.
Considering $\varphi$ as an element of $\ihom{S,T}$, the notation $\varphi(a)$ for $\counit_{S,T}(\varphi\otimes a)$ is consistent with the usual notation of $\varphi(a)$ for the evaluation of $\varphi$ at $a$. 
\end{rmk}

\begin{lma}
\label{prp:genProp:eval1}
Let $S$ be a \CuSgp.
Let $\ev_1\colon\CatCuMor[\NNbar,S]\to S$ be given by $\ev_1(f)=f(1)$ for $f\in\CatCuMor[\NNbar,S]$.
Then $\ev_1$ is an isomorphism of $\CatQ$-semigroups.
That is, $\ev_1$ is an additive order-isomorphism and we have $f\prec g$ if and only if $\ev_1(f)\ll\ev_1(g)$, for $f,g\in\CatCuMor[\NNbar,S]$.

It follows that $(\CatCuMor[\NNbar,S],\prec)$ is a \CuSgp{} (naturally isomorphic to $S$ via $\ev_1$).
Moreover, the endpoint map $\sigma_{\NNbar,S}\colon\ihom{\NNbar,S}\to\CatCuMor[\NNbar,S]$ from \autoref{dfn:endpoint} is an isomorphism.
\end{lma}
\begin{proof}
It is straightforward to prove that $\ev_1$ is an isomorphism of $\CatQ$-semigroups.
By \cite[Proposition 4.10]{AntPerThi17arX:AbsBivariantCu}, the endpoint map of a \CuSgp{} is an isomorphism.
Thus, the endpoint maps $\varphi_S$ and $\varphi_{\CatCuMor[\NNbar,S]}$ are isomorphisms.
By definition, $\sigma_{\NNbar,S} = \varphi_{\CatCuMor[\NNbar,S]}$.
Since $\ev_1$ is an isomorphism, so is $\tau(\ev_1)$.
\end{proof}

\begin{dfn}
\label{dfn:genProp:iIso}
Given a \CuSgp{} $S$, we let $\iIso_S\colon S\to\ihom{\NNbar,S}$ be the \CuMor{} that under the identification
\[
\CatCuMor\big( S, \ihom{\NNbar,S} \big)
\cong \CatCuMor\big( S\otimes\NNbar, S \big)
\]
corresponds to the natural isomorphism $\rtIso_S\colon S\otimes\NNbar\to S$.
\end{dfn}

We leave the proof of the following result to the reader.

\begin{prp}
\label{prp:genProp:iIso}
Let $S$ be a \CuSgp.
Then $\iIso_S\colon S\to\ihom{\NNbar,S}$ is an isomorphism.
The inverse of $\iIso_S$ is $\ev_1\circ\sigma_{\NNbar,S}$, where $\ev_1$ is evaluation at $1$ as in \autoref{prp:genProp:eval1}, and where $\sigma_{\NNbar,S}\colon\ihom{\NNbar,S}\to\CatCuMor[\NNbar,S]$ denotes the endpoint map from \autoref{dfn:endpoint}.
\end{prp}

We now introduce and study the external tensor product map. To this end, let first $S_k$ and $T_k$ be \CuSgp{s}, and let $\varphi_k\colon S_k\to T_k$ be (generalized) \CuMor{s}, for $k=1,2$. Recall from the comments after \cite[Theorem 2.10]{AntPerThi17arX:AbsBivariantCu} that the map $\varphi_1\times \varphi_2\colon S_1\times S_2 \to T_1\otimes T_2$, defined by
\[
(\varphi_1\times \varphi_2)(a_1,a_2):=f_1(a_1)\otimes f_2(a_2),
\]
for $a_1\in S_1$ and $a_2\in S_2$, is a (generalized) $\CatCu$-bimorphism.
We denote the induced (generalized) \CuMor{} by $\varphi_1 \otimes \varphi_2 \colon S_1\otimes S_2 \to T_1\otimes T_2$, and we call it the \emph{tensor product} of $\varphi_1$ and $\varphi_2$.

Next, we generalize this construction and define an external tensor product between elements of internal-homs.

\begin{dfn}
\label{dfn:genProp:tensExt}
Given \CuSgp{s} $S_1,S_2,T_1$ and $T_2$, we define the \emph{external tensor product map}
$\tensExt\colon \ihom{S_1,T_1} \otimes \ihom{S_2,T_2} \to \ihom{S_1\otimes S_2, T_1\otimes T_2}$ as the \CuMor{} that under the identification
\begin{align*}
&\CatCuMor\big( \ihom{S_1,T_1} \otimes \ihom{S_2,T_2}, \ihom{S_1\otimes S_2, T_1\otimes T_2} \big) \\
&\quad\cong \CatCuMor\big( \ihom{S_1,T_1} \otimes \ihom{S_2,T_2} \otimes S_1\otimes S_2, T_1\otimes T_2 \big),
\end{align*}
corresponds to the composition
\[
(\counit_{S_1,T_1}\otimes\counit_{S_2,T_2})\circ(\id_{\ihom{S_1,T_1}}\otimes\sigma\otimes\id_{S_2}),
\]
where $\sigma\colon \ihom{S_2,T_2}\otimes S_1 \to S_1\otimes\ihom{S_2,T_2}$ denotes the flip isomorphism.

Given $x_1\in\ihom{S_1,T_1}$ and $x_2\in\ihom{S_2,T_2}$, we denote the image of $x_1\otimes x_2$ under this map by $x_1\tensExt x_2$, and we call it the \emph{external tensor product }of $x_1$ and $x_2$.
\end{dfn}

\begin{rmk}
\label{rmk:genProp:tensExt}
Let $\varphi_1\colon S_1\to T_1$ and $\varphi_2\colon S_2\to T_2$ be \CuMor{s}.
Using \cite[Proposition 5.11]{AntPerThi17arX:AbsBivariantCu}, we identify $\varphi_1$ with a compact element in $\ihom{S_1,T_1}$, and similarly for $\varphi_2$.
It is easy to see that the element $\varphi_1\tensExt \varphi_2$ from \autoref{dfn:genProp:tensExt} agrees with the compact element in $\ihom{S_1\otimes S_2,T_1\otimes T_2}$ that is identified with the tensor product map $\varphi_1\otimes \varphi_2\colon S_1\otimes S_2\to T_1\otimes T_2$ from the comments before the above definition.

Notice that there is a certain ambiguity with the notation $\varphi_1\otimes\varphi_2$, in that it may refer to a \CuMor{} (identified with a compact element in $\ihom{S_1\otimes S_2,T_1\otimes T_2}$), and also to an element in $\ihom{S_1,T_1}\otimes\ihom{S_2,T_2}$.
However, the precise meaning will be clear from the context.
\end{rmk}

\begin{thm}
\label{prp:genProp:tensExt}
Let $S_1,S_2,T_1$ and $T_2$ be \CuSgp{s}, and let $\pathCu{f}=(f_\lambda)_\lambda$ and $\pathCu{g}=(g_\lambda)_\lambda$ be paths in $\CatCuMor[S_1,T_1]$ and $\CatCuMor[S_2,T_2]$, respectively.
For each $\lambda$, consider the generalized \CuMor{} $f_\lambda\otimes g_\lambda\colon S_1\otimes S_2 \to T_1\otimes T_2$.
Then $(f_\lambda\otimes g_\lambda)_\lambda$ is a path in $\CatCuMor[S_1\otimes S_2, T_1\otimes T_2]$ and we have
\[
[\pathCu{f}]\tensExt[\pathCu{g}]
= [(f_\lambda\otimes g_\lambda)_\lambda].
\]
\end{thm}
\begin{proof}
To show that $(f_\lambda\otimes g_\lambda)_\lambda$ is a path, let $\lambda',\lambda\in I_\QQ$ satisfy $\lambda'<\lambda$.
To show that $f_{\lambda'}\otimes g_{\lambda'}\prec f_\lambda\otimes g_\lambda$, let $t',t\in S_1\otimes S_2$ satisfy $t'\ll t$.
By properties of the tensor product in $\CatCu$, we can choose $n\in\NN$, elements $a_k',a_k\in S_1$ and $b_k',b_k\in S_2$ satisfying $a_k'\ll a_k$ and $b_k'\ll b_k$ for $k=1,\ldots,n$, and such that
\[
t' \leq \sum_{k=1}^n a_k'\otimes b_k',\andSep
\sum_{k=1}^n a_k\otimes b_k\leq t.
\]
We have $f_{\lambda'}\prec f_\lambda$ and $g_{\lambda'}\prec g_\lambda$, and therefore $f_{\lambda'}(a_k')\ll f_\lambda(a_k)$ and $g_{\lambda'}(b_k')\ll g_\lambda(b_k)$ for $k=1,\ldots,n$.
Using this at the third step we deduce that
\begin{align*}
(f_{\lambda'}\otimes g_{\lambda'})(t')
&\leq (f_{\lambda'}\otimes g_{\lambda'})\left( \sum_{k=1}^n a_k'\otimes b_k' \right) \\
&= \sum_{k=1}^n f_{\lambda'}(a_k')\otimes g_{\lambda'}(b_k') \\
&\ll \sum_{k=1}^n f_{\lambda}(a_k)\otimes g_{\lambda}(b_k) \\
&= (f_{\lambda}\otimes g_{\lambda})\left( \sum_{k=1}^n a_k\otimes b_k \right)
\leq (f_{\lambda}\otimes g_{\lambda})(t).
\end{align*}

Thus, given paths $\pathCu{p}=(p_\lambda)_\lambda$ and $\pathCu{q}=(q_\lambda)_\lambda$ in $\CatCuMor[S_1,T_1]$ and $\CatCuMor[S_2,T_2]$, respectively, then $(p_\lambda\otimes q_\lambda)_\lambda$ is a path in $\CatCuMor[S_1\otimes S_2,T_1\otimes T_2]$.
Moreover, it is tedious but straightforward to check that the map $\ihom{S_1,T_1}\times\ihom{S_2,T_2}\to\ihom{S_1\otimes S_2,T_1\otimes T_2}$ that sends a pair $([\pathCu{p}]\otimes[\pathCu{q}])$ to $[(p_\lambda\tensExt q_\lambda)_\lambda]$ is a well-defined $\CatCu$-bimorphism.
We let $\alpha\colon \ihom{S_1,T_1}\otimes\ihom{S_2,T_2}\to\ihom{S_1\otimes S_2,T_1\otimes T_2}$ be the induced $\CatCu$-morphism.

To show that $[\pathCu{f}]\tensExt[\pathCu{g}]=[(f_\lambda\otimes g_\lambda)_\lambda]$, we will prove that the external tensor product $\tensExt$ and the map $\alpha$ correspond to the same $\CatCu$-morphism under the bijection
\begin{align*}
&\CatCuMor\big( \ihom{S_1,T_1}\otimes\ihom{S_2,T_2}, \ihom{S_1\otimes S_2,T_1\otimes T_2} \big) \\
&\quad\cong \CatCuMor\big( \ihom{S_1,T_1}\otimes\ihom{S_2,T_2}\otimes S_1\otimes S_2, T_1\otimes T_2 \big)
\end{align*}
from \autoref{thm:Cuclosedbijection}.

Let $\pathCu{p}=(p_\lambda)_\lambda$ and $\pathCu{q}=(q_\lambda)_\lambda$ be paths in $\CatCuMor[S_1,T_1]$ and $\CatCuMor[S_2,T_2]$, respectively, and let $s_i$ be elements in $S_i$, for $i=1,2$.
By definition of $\tensExt$ (see \autoref{dfn:genProp:tensExt} and \autoref{ntn:ihomisomorphism}), we have
\[
\bar{\tensExt}([\pathCu{p}] \otimes [\pathCu{q}] \otimes s_1 \otimes s_2)
= [\pathCu{p}](s_1)\otimes [\pathCu{q}](s_2)
= p_1(s_1)\otimes q_1(s_2).
\]
Using \autoref{thm:Cuclosedbijection} at the first step, we obtain that
\begin{align*}
\bar{\alpha}([\pathCu{p}] \otimes [\pathCu{q}]\otimes s_1 \otimes s_2)
&=\sigma_{S_1\otimes S_2,T_1\otimes T_2}(\alpha([\pathCu{p}]\otimes [\pathCu{q}]))(s_1\otimes s_2) \\
&=\sigma_{S_1\otimes S_2,T_1\otimes T_2}([(p_\lambda\otimes q_\lambda)_\lambda])(s_1\otimes s_2) \\
&=(p_1\otimes q_1)(s_1\otimes s_2)
= p_1(s_1)\otimes q_1(s_2).
\end{align*}
It follows that $\tensExt=\alpha$ and therefore
\[
[\pathCu{f}]\tensExt[\pathCu{g}]
= \tensExt([\pathCu{f}]\otimes[\pathCu{g}])
= \alpha([\pathCu{f}]\otimes[\pathCu{g}])
= [(f_\lambda\otimes g_\lambda)_\lambda]. \qedhere
\]
\end{proof}

The following result shows that the external tensor product is associative.

\begin{prp}
\label{prp:genProp:tensExtAssoc}
Let $S_1,S_2,T_1,T_2,P_1$ and $P_2$ be \CuSgp{s}, let $x\in\ihom{S_1,S_2}$, $y\in\ihom{T_1,T_2}$, and let $z\in\ihom{P_1,P_2}$.
For $k=1,2$, we identify $(S_k\otimes T_k)\otimes P_k$ with $S_k\otimes (T_k \otimes P_k)$ using the natural isomorphism from the monoidal structure of $\CatCu$ (see comments after \cite[Theorem 2.10]{AntPerThi17arX:AbsBivariantCu}).
Then
\[
(x\tensExt y)\tensExt z = x\tensExt(y\tensExt z).
\]
\end{prp}
\begin{proof}
Given $f\in\CatCuMor[S_1,S_2]$, $g\in\CatCuMor[T_1,T_2]$ and $h\in\CatCuMor[P_1,P_2]$, it is straightforward to check that
\[
(f\otimes g)\otimes h = f\otimes(g\otimes h),
\]
as generalized \CuMor{s} $S_1\otimes T_1\otimes P_1\to S_2\otimes T_2\otimes P_2$.
The result follows by applying \autoref{prp:genProp:tensExt}.
\end{proof}

\begin{pbm}
\label{pbm:genProp:tensExt}
Study the order-theoretic properties of the external tensor product map $\tensExt\colon \ihom{S_1,T_1} \otimes \ihom{S_2,T_2} \to \ihom{S_1\otimes S_2, T_1\otimes T_2}$.
In particular, when is this map an order-embedding, when is it surjective?
\end{pbm}

We recall below the definition of the composition product and analyse its relation with the external tensor product.

\begin{dfn}
\label{dfn:genProp:composition}
Given \CuSgp{s} $S$, $T$ and $P$, we define the \emph{composition product}
\[
\circ\colon \ihom{T,P} \otimes \ihom{S,T} \to \ihom{S,P}
\]
as the \CuMor{} that under the identification
\[
\CatCuMor\big( \ihom{T,P} \otimes \ihom{S,T}, \ihom{S,P} \big)
\cong \CatCuMor\big( \ihom{T,P} \otimes \ihom{S,T} \otimes S, P \big)
\]
corresponds to the composition $\counit_{T,P}\circ(\id_{\ihom{T,P}}\otimes\counit_{S,T})$.
Given $x\in\ihom{S,T}$ and $y\in\ihom{T,P}$, we denote the image of $y\otimes x$ under the composition product by $y\circ x$.

Given $x\in\ihom{S,T}$, we let $x^*\colon\ihom{T,P}\to\ihom{S,P}$ be given by $x^*(y):=y\circ x$ for $y\in\ihom{T,P}$.
Analogously, given $y\in\ihom{T,P}$, we let $y_*\colon\ihom{S,T}\to\ihom{S,P}$ be given by $y_*(x):=y\circ x$ for $x\in\ihom{S,T}$.
\end{dfn}


The composition product for the internal-hom satisfies the axioms for the hom sets in an enriched category. In particular, the product is associative and  
the identity element $\id_S\in\CatCuMor(S,S)\subseteq\ihom{S,S}$ acts as a unit for the composition product (see \cite[Section 1.6]{Kel05EnrichedCat}). We recall these facts in the following proposition:

\begin{prp}
\label{prp:genProp:compAssoc}
Let $S,T,P$ and $Q$ be \CuSgp{s}, let $x\in\ihom{S,T}$, $y\in\ihom{T,P}$, and let $z\in\ihom{P,Q}$.
Then
\[
(z\circ y)\circ x = z\circ(y\circ x).
\]
Further, for the identity \CuMor{s} $\id_S\in\CatCuMor(S,S)$ and $\id_T\in\CatCuMor(T,T)$, we have
\[
\id_T\circ x = x = x\circ\id_S.
\]
\end{prp}


It follows that $\ihom{S,S}$ and $\ihom{T,T}$ are (not necessarily commutative) $\CatCu$-semirings and that $\ihom{S,T}$ has a natural left $\ihom{S,S}$- and right $\ihom{T,T}$-semimodule structure; see Propositions~\ref{prp:ring:ihomSS} and~\ref{prp:ring:ihomST} in the next section. In the following proposition we give an explicit description of the composition product.

\begin{prp}
\label{prp:genProp:composition}
Let $S,T,P$ be \CuSgp{s}, and let $\pathCu{f}=(f_\lambda)_\lambda$ and $\pathCu{g}=(g_\lambda)_\lambda$ be paths in $\CatCuMor[S,T]$ and $\CatCuMor[T,P]$, respectively.
For each $\lambda$, consider the generalized \CuMor{} $g_\lambda\circ f_\lambda\colon S\to P$.
Then $(g_\lambda\circ f_\lambda)_\lambda$ is a path in $\CatCuMor[S,P]$ and
\[
[\pathCu{g}]\circ[\pathCu{f}]
= [(g_\lambda\circ f_\lambda)_\lambda].
\]
\end{prp}
\begin{proof}
It is easy to check that $(g_\lambda\circ f_\lambda)_\lambda$ is a path.
Moreover, it is tedious but straightforward to check that the map $\ihom{T,P}\times\ihom{S,T}\to\ihom{S,P}$ that sends a pair $([\pathCu{p}],[\pathCu{q}])$ to $[(q_\lambda\circ p_\lambda)_\lambda]$ is a well-defined $\CatCu$-bimorphism.
We let $\alpha\colon \ihom{T,P}\otimes\ihom{S,T}\to\ihom{S,P}$ be the induced \CuMor.

To show that $[\pathCu{g}]\circ[\pathCu{f}]=[(g_\lambda\otimes f_\lambda)_\lambda]$, we will prove that the composition product $\circ$ and the map $\alpha$ correspond to the same \CuMor{} under the bijection
\[
\CatCuMor\big( \ihom{T,P} \otimes \ihom{S,T}, \ihom{S,P} \big)
\cong \CatCuMor\big( \ihom{T,P} \otimes \ihom{S,T} \otimes S, P \big)
\]
from \autoref{thm:Cuclosedbijection}.

Let $\pathCu{p}=(p_\lambda)_\lambda$ and $\pathCu{q}=(q_\lambda)_\lambda$ be paths in $\CatCuMor[S,T]$ and $\CatCuMor[T,P]$, respectively, and let $s\in S$.
Set $p_1:=\sup\limits_{\lambda<1} p_\lambda$ and $q_1:=\sup\limits_{\lambda<1}q_\lambda$.
By definition, we have
\begin{align*}
\bar{\circ}([\pathCu{q}] \otimes [\pathCu{p}] \otimes s)
&= \counit_{T,P}\circ(\id_{\ihom{T,P}}\otimes\counit_{S,T}) ([\pathCu{q}] \otimes [\pathCu{p}] \otimes s) 
= \counit_{T,P}([\pathCu{q}]\otimes \counit_{S,T}([\pathCu{p}] \otimes s)) \\
&= \counit_{T,P}([\pathCu{q}]\otimes p_1(s))
= q_1(p_1(s)).
\end{align*}
On the other hand, using \autoref{thm:Cuclosedbijection} at the first step, we obtain that
\begin{align*}
\bar{\alpha}([\pathCu{q}] \otimes [\pathCu{p}] \otimes s)
&=\sigma_{S,P}(\alpha([\pathCu{q}] \otimes [\pathCu{p}]))(s) 
=\sigma_{S,P}([(q_\lambda\circ p_\lambda)_\lambda])(s) \\
&=(q_1\circ p_1)(s)
= q_1(p_1(s)).
\end{align*}
It follows that $\circ=\alpha$ and therefore
\[
[\pathCu{g}]\circ[\pathCu{f}]
= \circ([\pathCu{g}]\otimes[\pathCu{f}])
= \alpha([\pathCu{g}]\otimes[\pathCu{f}])
= [(g_\lambda\otimes f_\lambda)_\lambda]. \qedhere
\]
\end{proof}

Note that, in \autoref{prp:genProp:composition}, the composition product of two \CuMor{s}, viewed as compact elements in the internal-hom set, is the usual composition of morphisms as maps.

Next we show that the composition product is compatible with the evaluation map in the expected way.
It will follow later that the evaluation map $\counit_{S,S}\colon\ihom{S,S}\otimes S\to S$ defines a natural left $\ihom{S,S}$-semimodule structure on $S$;
see \autoref{prp:ring:ihomSS_act_S}.

\begin{lma}
\label{prp:genProp:endpointMapAssociative}
Let $S,T$ and $P$ be \CuSgp{s}, let $x\in\ihom{S,T}$, and let $y\in\ihom{T,P}$.
Then
\[
\sigma_{S,P}(y\circ x) = \sigma_{T,P}(y)\circ\sigma_{S,T}(x).
\]
\end{lma}
\begin{proof}
Let $\pathCu{f}=(f_\lambda)_\lambda$ be a path in $\CatCuMor[S,T]$ representing $x$, and let $\pathCu{g}=(g_\lambda)_\lambda$ be a path in $\CatCuMor[T,P]$ representing $y$.
Let $a\in S$.
By \autoref{prp:genProp:composition}, we have $y\circ x=[(g_\lambda\circ f_\lambda)_\lambda]$.
Using this at the first step, we obtain that
\[
\sigma_{S,P}(y\circ x)(a)
= \sup_{\lambda\in I_\QQ} (g_\lambda\circ f_\lambda)(a)
= \sup_{\mu\in I_\QQ} g_\mu \left( \sup_{\lambda\in I_\QQ} f_\lambda(a) \right)
= \sigma_{T,P}(y)\left( \sigma_{S,T}(x)(a) \right),
\]
as desired.
\end{proof}

By combining \autoref{prp:genProp:endpointMapAssociative} with \autoref{prp:genProp:counitMap}, we obtain:

\begin{prp}
\label{prp:genProp:comp_counit}
Let $S,T$ and $P$ be \CuSgp{s}, let $x\in\ihom{S,T}$, let $y\in\ihom{T,P}$, and let $a\in S$.
Then
\[
(y\circ x)(a) = y(x(a)).
\]
Moreover, for the identity \CuMor{} $\id_S\in\CatCuMor(S,S)$, we have $\id_S(a)=a$.
\end{prp}

The following result shows that the external tensor product and the composition product commute.

\begin{prp}
\label{prp:genProp:tensExtComp}
Let $S_1,S_2,T_1$ and $T_2$ be \CuSgp{s}.
Given $x_k\in\ihom{S_k,T_k}$ and $y_k\in\ihom{T_k,P_k}$ for $k=1,2$, we have
\[
(y_2\tensExt y_1)\circ(x_2\tensExt x_1)
= (y_2\circ x_2)\tensExt(y_1\circ x_1).
\]
\end{prp}
\begin{proof}
Let $\pathCu{f^{(k)}}=(f^{(k)}_\lambda)_\lambda$ be a path in $\CatCuMor[S_k,T_k]$ representing $x_k$, for $k=1,2$, and let $\pathCu{g^{(k)}}=(g^{(k)}_\lambda)_\lambda$ be a path in $\CatCuMor[T_k,P_k]$ representing $y_k$, for $k=1,2$.
Given $\lambda$, it is straightforward to check that
\[
(g^{(2)}_\lambda\otimes g^{(1)}_\lambda) \circ (f^{(2)}_\lambda\otimes f^{(1)}_\lambda)
= (g^{(2)}_\lambda \circ f^{(2)}_\lambda) \otimes (g^{(1)}_\lambda\circ f^{(1)}_\lambda).
\]
Using this at the second step, and using \autoref{prp:genProp:tensExt} and \autoref{prp:genProp:composition} at the first and last step, we obtain that
\begin{align*}
(y_2\tensExt y_1)\circ(x_2\tensExt x_1)
&= \left[ \big( (g^{(2)}_\lambda\otimes g^{(1)}_\lambda) \circ (f^{(2)}_\lambda\otimes f^{(1)}_\lambda) \big)_\lambda \right] \\
&= \left[ \big( (g^{(2)}_\lambda \circ f^{(2)}_\lambda) \otimes (g^{(1)}_\lambda\circ f^{(1)}_\lambda) \big)_\lambda \right] \\
&= (y_2\circ x_2)\tensExt(y_1\circ x_1). \qedhere
\end{align*}
\end{proof}

In the last part of this section, we revisit the unit and counit maps, their functorial properties, and how they can be used to implement the adjuntion between the tensor product and the internal-hom functors.

\begin{prp}
\label{prp:genProp:correspondence}
Let $S,T$ and $P$ be \CuSgp{s}.
Then the bijection
\[
\CatCuMor\big( S,\ihom{T,P} \big) \cong \CatCuMor\big( S\otimes T,P \big)
\]
from \autoref{thm:Cuclosedbijection} identifies a \CuMor{} $f\colon S\to\ihom{T,P}$ with
\[
\counit_{T,P}\circ(f\otimes\id_T) \colon S\otimes T \xrightarrow{f\otimes\id_T} \ihom{T,P}\otimes T \xrightarrow{\counit_{T,P}} P.
\]
Conversely, a \CuMor{} $g\colon S\otimes T\to P$ is identified with
\[
g_\ast\circ\unit_{S,T} \colon S\xrightarrow{\unit_{S,T}} \ihom{T,S\otimes T} \xrightarrow{g_\ast} \ihom{T,P}.
\]
In particular, we have
\begin{align*}
f = \left( \counit_{T,P}\circ(f\otimes\id_T) \right)_\ast\circ \unit_{S,T},\andSep
g = \counit_{T,P}\circ((g_\ast\circ\unit_{S,T})\otimes\id_T).
\end{align*}
\end{prp}
\begin{proof}
Let $f\colon S\to\ihom{T,P}$ be a \CuMor.
Under the natural bijection from \autoref{thm:Cuclosedbijection}, $f$ corresponds to the \CuMor{} $\bar{f}\colon S\otimes T\to P$ with
\[
\bar{f}(s\otimes t)=\sigma_{T,P}(f(s))(t),
\]
for a simple tensor $s\otimes t\in S\otimes T$.
On the other hand, we have
\[
(\counit_{T,P}\circ(f\otimes\id_T))(s\otimes t)
= \counit_{T,P}(f(s) \otimes t)
= \sigma_{T,P}(f(s))(t),
\]
for a simple tensor $s\otimes t\in S\otimes T$.
Thus $\bar{f}$ and $\counit_{T,P}\circ(f\otimes\id_T)$ agree on simple tensors, and consequently $\bar{f}=\counit_{T,P}\circ(f\otimes\id_T)$, as desired.

Let $g\colon S\otimes T\to P$ be a \CuMor.
Set $\alpha:=g_\ast\circ\unit_{S,T}$.
Under the natural bijection from \autoref{thm:Cuclosedbijection}, $\alpha$ corresponds to the \CuMor{} $\bar{\alpha}\colon S\otimes T\to P$ with
\[
\bar{\alpha}(s\otimes t)=\sigma_{T,P}(\alpha(s))(t),
\]
for a simple tensor $s\otimes t\in S\otimes T$.
It is straightforward to verify that $\sigma_{T,P}\circ g_\ast = g_\ast\circ\sigma_{T,S\otimes T}$.
Using this at the third step, and using \autoref{prp:genProp:unitMapEndpoint} at the fourth step, we deduce that
\begin{align*}
\bar{\alpha}(s\otimes t)
=\sigma_{T,P}(\alpha(s))(t)
&=(\sigma_{T,P}\circ g_\ast\circ\unit_{S,T})(s)(t) \\
&=(g_\ast\circ\sigma_{T,S\otimes P}\circ\unit_{S,T})(s)(t)
=g(s\otimes t),
\end{align*}
for every simple tensor $s\otimes t\in S\otimes T$.
Thus, $\bar{\alpha}=g$, as desired.
\end{proof}

Applying the previous result to the identity morphisms, we obtain:

\begin{cor}
\label{prp:genProp:unit_counit}
Let $S$ and $T$ be \CuSgp{s}.
Then
\begin{align*}
\id_{\ihom{S,T}} = (\counit_{S,T})_\ast\circ\unit_{\ihom{S,T},S}, \andSep
\id_{S\otimes T} = \counit_{T,S\otimes T}\circ(\unit_{S,T}\otimes\id_T).
\end{align*}
\end{cor}

Given \CuSgp{s} $S$ and $T$, we consider the unit map $\unit_{S,T}\colon S\to\ihom{T,S\otimes T}$ from \autoref{dfn:genProp:unitMap}.
Next, we introduce a more general form of the unit map.

\begin{dfn}
\label{dfn:genProp:genUnit}
Let $S,T$ and $T'$ be \CuSgp{s}.
We define \emph{the general left unit map} $S\otimes\ihom{T',T} \to \ihom{T', S\otimes T}$ as the \CuMor{} that under the identification
\[
\CatCuMor\big( S\otimes\ihom{T',T}, \ihom{T', S\otimes T} \big)
\cong \CatCuMor\big( S\otimes\ihom{T',T}\otimes T', S\otimes T \big)
\]
corresponds to the map $\id_S\otimes\counit_{T',T}$.
Given $a\in S$ and $x\in\ihom{T',T}$, we denote the image of $a\otimes x$ under this map by $\prescript{}{a}{x}$.

Analogously, we define \emph{the general right unit map}
$\ihom{T',T}\otimes S \to \ihom{T', T\otimes S}$ as the \CuMor{} that under the identification
\[
\CatCuMor\big( \ihom{T',T}\otimes S, \ihom{T', T\otimes S} \big)
\cong \CatCuMor\big( \ihom{T',T}\otimes S\otimes T', T\otimes S \big)
\]
corresponds to the map $(\counit_{T',T}\otimes\id_S)\circ(\id_{\ihom{T',T}}\otimes\sigma)$, where $\sigma$ denotes the flip isomorphism.
Given $a\in S$ and $x\in\ihom{T',T}$, we denote the image of $x\otimes a$ under this map by $x_a$.
\end{dfn}

We leave the proof of the following result to the reader.

\begin{prp}
\label{prp:genProp:genUnit}
Let $S,T$ and $T'$ be \CuSgp{s}, let $a$ be an element in $S$, and let $x$ be an element in $\ihom{T',T}$.
Let $\iIso_S\colon S\to \ihom{\NNbar, S}$ be the isomorphism from \autoref{dfn:genProp:iIso},
and let $\ltIso_{T'}\colon \NNbar\otimes T'\to T'$ and $\rtIso_{T'}\colon T'\otimes\NNbar \to T'$ be the natural $\CatCu$-isomorphism.
Then
\[
\prescript{}{a}{x}
= (\iIso_S(a)\otimes x)\circ \ltIso_{T'}^{-1}
= \unit_{S,T}(a)\circ x
= (\id_S\otimes x)\circ \unit_{S,T'}(a).
\]
and analogously $x_a = (x\otimes\iIso_S(a))\circ \rtIso_{T'}^{-1}$.
Further, for the unit map $\unit_{S,T}\colon S\to\ihom{T,S\otimes T}$, we have $\unit_{S,T}(a) = \prescript{}{a}{(\id_T)}$ for every $a\in S$.
\end{prp}

%

Finally, similar to $KK$-theory for \ca{s}, we have a general form of the product that simultaneously generalizes the composition product and the external tensor product; see \cite[Section~18.9, p.180f]{Bla98KThy}.

Let $P$, $S_1$, $S_2$, $T_1$ and $T_2$ be \CuSgp{s}.
We let
\[
\tensExt_P\colon \ihom{S_1 \otimes P, T_1} \otimes \ihom{S_2, P \otimes T_2} \to \ihom{S_1 \otimes S_2, T_1 \otimes T_2},
\]
be the \CuMor{} that under the identification
\begin{align*}
&\CatCuMor\big( \ihom{S_1 \otimes P, T_1} \otimes \ihom{S_2, P \otimes T_2}, \ihom{S_1 \otimes S_2, T_1 \otimes T_2} \big) \\
&\quad\cong \CatCuMor\big( \ihom{S_1 \otimes P, T_1} \otimes \ihom{S_2, P \otimes T_2} \otimes S_1 \otimes S_2, T_1\otimes T_2 \big)
\end{align*}
corresponds to the composition
\[
(\counit_{S_1\otimes P,T_1}\otimes\id_{T_2})
\circ (\id_{\ihom{S_1\otimes P,T_1}}\otimes\counit_{S_2,P\otimes T_2})
\circ (\id_{\ihom{S_1\otimes P,T_1}}\otimes\sigma_{\ihom{S_2, P \otimes T_2},S_1}\otimes\id_{S_2}),
\]
where $\sigma_{\ihom{S_2, P \otimes T_2},S_1}$ denotes the flip isomorphism.

Given $x\in\ihom{S_1\otimes P, T_1}$ and $y\in\ihom{S_2, P\otimes T_2}$, we have
\[
x\tensExt_P y = (x\tensExt\id_{T_2}) \circ (\id_{S_1}\tensExt y).
\]

Specializing to the case $P=\NNbar$, we obtain the external tensor product, after applying the usual isomorphisms $S_1\otimes\NNbar\cong S_1$ and $\NNbar\otimes T_2\cong T_2$.

Specializing to the case $T_2=S_1=\NNbar$, we obtain the composition product, after applying the natural isomorphisms $\NNbar\otimes P \cong P \cong P\otimes\NNbar$, $\NNbar\otimes S_2\cong S_2$, and $T_1\otimes\NNbar\cong T_1$.

\section{Bivariant Cuntz semigroups of ideals and quotients}
\label{sec:bivarCu:functoriality}


A \emph{sub-\CuSgp} of a \CuSgp{} $T$ is a submonoid $S\subseteq T$ that is a \CuSgp{} for the partial order inherited from $T$ and such that the inclusion $S\to T$ is a \CuMor.
It is easy to see that $S$ is a sub-\CuSgp{} of $T$ if and only if $S$ is closed under passing to suprema of increasing sequences and if the way-below relation in $S$ and $T$ agree.

\begin{lma}
\label{prp:bivarCu:inclInducesOEmb}
Let $S$ and $T$ be \CuSgp{s}, and let $T'\subseteq T$ be a sub-\CuSgp{}.
Then the inclusion map $\iota\colon T'\to T$ induces an order-embedding $\iota_*\colon \ihom{S,T'} \to \ihom{S,T}$.
\end{lma}
\begin{proof}
Let $\pathCu{f}=(f_\lambda)_\lambda$ be a path in $\CatCuMor[S,T']$.
Then $\pathCu{\tilde{f}}:=(\iota\circ f_\lambda)_\lambda$ is a path in $\CatCuMor[S,T]$ and we have $\iota_*([\pathCu{f}])=[\pathCu{\tilde{f}}]$;
see the comments after Remark~5.4 in \cite{AntPerThi17arX:AbsBivariantCu}.
	
To show that $\iota_*$ is an order-embedding, let $x,y\in\ihom{S,T'}$ with $\iota_*(x)\leq\iota_*(y)$.
Choose paths $\pathCu{f}$ and $\pathCu{g}$ in $\CatCuMor[S,T']$ representing $x$ and $y$, respectively.
We have $(\iota\circ f_\lambda)_\lambda \precsim (\iota\circ g_\lambda)_\lambda$.
Thus, for every $\lambda\in I_\QQ$, there is $\mu\in I_\QQ$ such that $\iota\circ f_\lambda \prec \iota\circ g_\mu$.
Using that $T'\subseteq T$ is a sub-\CuSgp{}, for such $\lambda$ and $\mu$ we deduce that $f_\lambda \prec g_\mu$.
(We use that for $a',a\in T'$ we have $a'\ll a$ in $T'$ if and only if $\iota(a')\ll\iota(a)$ in $T$.)
It follows that $\pathCu{f}\precsim\pathCu{g}$, and hence $x\leq y$, as desired.
\end{proof}

Recall that an \emph{ideal} of a \CuSgp{} $S$ is a submonoid $J\subseteq S$ that is closed under passing to suprema of increasing sequences and that is downward-hereditary.
Every ideal is in particular a sub-\CuSgp{}.
(See \cite[Section~5.1]{AntPerThi18:TensorProdCu} for an account on ideals and quotients.)

\begin{prp}
\label{prp:bivarCu:idealSecondEntry}
Let $S$ and $T$ be \CuSgp{s}, and let $J$ be an ideal of $T$.
Let $\iota\colon J\to T$ denote the inclusion map.
Then the induced \CuMor{} $\iota_*\colon\ihom{S,J}\to\ihom{S,T}$ is an order-embedding that identifies $\ihom{S,J}$ with an ideal of $\ihom{S,T}$.
Moreover, $x\in\ihom{S,T}$ belongs to $\ihom{S,J}$ if and only if for some (equivalently, for every) path $(f_\lambda)_\lambda$ representing $x$, each $f_\lambda$ takes image in $J$.
\end{prp}
\begin{proof}
By \autoref{prp:bivarCu:inclInducesOEmb}, $\iota_*$ is an order-embedding.
Hence, $\iota_*$ identifies $\ihom{S,J}$ with a submonoid of $\ihom{S,T}$ that is closed under passing to suprema of increasing sequences.

Let $x\in\ihom{S,T}$ be represented by a path $\pathCu{f}=(f_\lambda)_\lambda$ in $\CatCu[S,T]$.
If each $f_\lambda$ takes values in $J$, then we can consider $\pathCu{f}$ as a path in $\CatCu[S,J]$ whose class is an element $x'\in\ihom{S,J}$ satisfying $\iota_*(x')=x$.
Conversely, assume that $x$ belongs to $\ihom{S,J}$.
Then there is a path $\pathCu{g}=(g_\mu)_\mu$ in $\CatCuMor[S,J]$ with $\iota_*([\pathCu{g}])=x$.
Let $\lambda\in I_\QQ$.
Since $\pathCu{f}\precsim\pathCu{g}$, we can choose $\mu\in I_\QQ$ with $f_\lambda\prec g_\mu$.
Since $g_\mu$ takes values in $J$, and since $J$ is downward-hereditary, it follows that $f_\lambda$ takes values in $J$, as desired.

A similar argument shows that $\ihom{S,J}$ is downward-hereditary in $\ihom{S,T}$.
\end{proof}

\begin{pgr}
\label{pgr:bivarCu:exactSecondEntry}
Given $S$, let us study whether the functor $\ihom{S,\freeVar}\colon\CatCu\to\CatCu$ is exact.
More precisely, let $J\lhd T$ be an ideal, with inclusion map $\iota\colon J\to T$ and with quotient map $\pi\colon T\to T/J$.
This induces the following \CuMor{s}:
\[
\ihom{S,J} \xrightarrow{\iota_*} \ihom{S,T} \xrightarrow{\pi_*} \ihom{S,T/J}.
\]
By \autoref{prp:bivarCu:idealSecondEntry}, $\iota_*$ identifies $\ihom{S,J}$ with an ideal in $\ihom{S,T}$. 
Since $\pi\circ\iota$ is the zero map, so is $\pi_*\circ\iota_*$.
Thus, $\pi_*$ vanishes on the ideal $\ihom{S,J}\lhd\ihom{S,T}$.
It follows that $\pi_*$ induces a \CuMor{}
\[
\widehat{\pi_*}\colon\ihom{S,T}/\ihom{S,J} \to \ihom{S,T/J}.
\]
\end{pgr}

\begin{pbm}
\label{pbm:bivarCu:exactSecondEntry}
Study the order-theoretic properties of the \CuMor{} $\widehat{\pi_*}$ from \autoref{pgr:bivarCu:exactSecondEntry}.
In particular, when is $\widehat{\pi_*}$ an order-embedding, when is it surjective?
\end{pbm}

We are currently not aware of any example for $S$ and $J\lhd T$ such that the map $\widehat{\pi_*}\colon\ihom{S,T}/\ihom{S,J} \to \ihom{S,T/J}$ is not an isomorphism.

The following result and its proof are analogous to \autoref{prp:bivarCu:idealSecondEntry}.

\begin{prp}
\label{prp:bivarCu:idealFirstEntry}
Let $S$ and $T$ be \CuSgp{s}, let $J\lhd S$, and let $\pi\colon S\to S/J$ denote the quotient map.
Then the induced \CuMor{} $\pi^*\colon\ihom{S/J,T}\to\ihom{S,T}$ is an order-embedding that identifies $\ihom{S/J,T}$ with an ideal in $\ihom{S,T}$.
Moreover, $x\in\ihom{S,T}$ belongs to $\ihom{S/J,T}$ if and only if for some (equivalently, for every) path $(f_\lambda)_\lambda$ representing $x$, each $f_\lambda$ vanishes on $J$.
\end{prp}

\begin{pgr}
\label{pgr:bivarCu:LatIhomST}
By Propositions~\ref{prp:bivarCu:idealSecondEntry} and~\ref{prp:bivarCu:idealFirstEntry}, ideals in $S$ and $T$ naturally induce ideals in $\ihom{S,T}$.
More precisely, if $J\lhd S$ and $K\lhd T$, then we can identify $\ihom{S/J,K}$ with an ideal in $\ihom{S,T}$.
Let $\Lat(P)$ denote the ideal lattice of a $\CatCu$-semigroup $P$.
We obtain a natural map
\[
\Lat(S)^{\mathrm{op}}\times\Lat(T)\to\Lat(\ihom{S,T}).
\]

However, this map need not be injective.
For example, consider $S=Z$ and $T=\NNbar\oplus Z$ with the ideal $J=0\oplus Z$.
Note that every generalized \CuMor{} $Z\to\NNbar\oplus Z$ necessarily takes values in the ideal $0\oplus Z$.
It follows that in this case $\ihom{S,J}=\ihom{S,T}$.

The following example shows that the above map is also not surjective in general.
In fact, the example shows that there exists a simple \CuSgp{} $S$ such that $\ihom{S,S}$ is not simple.
\end{pgr}

\begin{exa}
\label{exa:bivarCu:ihomSSNotSimple}
Let $S:=[0,1]\cup\{\infty\}$, considered with order and addition as a subset of $\PPbar$, with the convention that $a+b=\infty$ whenever $a+b>1$ in $\PPbar$.
It is easy to check that $S$ is a simple \CuSgp{}.

Given $t\in\{0\}\cup[1,\infty]$, let $\varphi_t\colon S\to S$ be the map given by $\varphi_t(a):=ta$, where $ta$ is given by the usual multiplication in $\PPbar$ applying the above convention that an element is $\infty$ as soon as it is larger than $1$.
Then $\varphi_t$ is a generalized \CuMor.
One can show that every generalized \CuMor{} $S\to S$ is of this form.
Hence, $\CatCuMor[S,S]$ is isomorphic to $\{0\}\cup[1,\infty]$, identifying $\prec$ with $\leq$.
It follows that
\[
\ihom{S,S}
= \tau\big( \CatCuMor[S,S],\prec \big)
\cong \tau\big( \{0\}\cup[1,\infty],\leq \big)
\cong \{0\}\sqcup[1,\infty] \sqcup (1,\infty],
\]
which is a disjoint union of compact elements corresponding to $\{0\}\cup[1,\infty]$ and nonzero soft elements corresponding to $(1,\infty]$.
(Similar to the decomposition of $Z$ and $R_q$.)
In particular, $\ihom{S,S}$ contains a compact infinite element $\infty$, and a noncompact infinite element $\infty'$.
The set $J:=\{x : x\leq\infty'\}$ is an ideal in $\ihom{S,S}$.
We have $\infty\notin J$, which shows that $\ihom{S,S}$ is not simple.
\end{exa}

\begin{pbm}
Characterize when $\ihom{S,T}$ is simple.
In particular, given simple \CuSgp{s} $S$ and $T$, give necessary and sufficient criteria for $\ihom{S,T}$ to be simple.
\end{pbm}

\section{\texorpdfstring{$\CatCu$}{Cu}-semirings and \texorpdfstring{$\CatCu$}{Cu}-semimodules}
\label{sec:ring}

A (unital) \emph{\CuSrg{}} is a \CuSgp{} $R$ together with a $\CatCu$-bimorphism $R\times R\to R$, denoted by $(r_1,r_2)\mapsto r_1r_2$, and a distinguished element $1\in R$, called the unit of $R$, such that $r_1(r_2r_3)=(r_1r_2)r_3$ and $r1=r=1r$ for all $r, r_1, r_2, r_3\in R$.
This concept was introduced and studied in \cite[Chapter~7]{AntPerThi18:TensorProdCu}, where it is further assumed that the product is commutative.
We will not make this assumption here.
We let $\mu_R\colon R\otimes R\to R$ denote the \CuMor{} induced by multiplication in $R$.

The following result follows from the general properties of the composition product for the internal-hom (see \cite[Section 1.6]{Kel05EnrichedCat}).
 
\begin{prp}
\label{prp:ring:ihomSS}
Let $S$ be a \CuSgp{}.
Then $\ihom{S,S}$ is a \CuSrg{} with product given by the composition product $\circ\colon\ihom{S,S}\otimes\ihom{S,S}\to\ihom{S,S}$, and with unit element given by the identity map $\id_S\in\ihom{S,S}$.
\end{prp}

\begin{rmk}
\label{rmk:ring:ihomSS}
Let $S$ be a \CuSgp.
The identity map $\id_S\colon S\to S$ is a \CuMor.
Therefore, the unit of the \CuSrg{} $\ihom{S,S}$ is compact.

In \autoref{exa:ring:ihomSS_noncomm}, we will see that $\ihom{S,S}$ is noncommutative in general.
\end{rmk}

Given a \CuSrg{} $R$, a \emph{left $\CatCu$-semimodule} over $R$ is a \CuSgp{} $S$ together with a $\CatCu$-bimorphim $R\times S\to S$, denoted by $(r,a)\mapsto ra$, such that for all $r_1,r_2\in R$ and $a\in S$, we have $(r_1r_2)a=r_1(r_2a)$ and $1a=a$.
We also say that $S$ has a \emph{left action} of $R$ if $S$ is a left $\CatCu$-semimodule over $R$.
Right $\CatCu$-semimodules are defined analogously.
If $R_1$ and $R_2$ are \CuSrg{s}, we say that a \CuSgp{} $S$ is a \emph{$(R_1,R_2)$-$\CatCu$-semibimodule} if it has a left $R_1$-action and a right $R_2$-action that satisfy $r_1(ar_2)=(r_1a)r_2$ for all $r_1\in R_1$, $r_2\in R_2$ and $a\in S$.

We refer the reader to \cite[Chapter~7]{AntPerThi18:TensorProdCu} for a discussion on commutative $\CatCu$-semirings and their $\CatCu$-semimodules.

\begin{prp}
\label{prp:ring:ihomSS_act_S}
Let $S$ be a \CuSgp{}.
Then $\counit_{S,S}\colon\ihom{S,S}\otimes S\to S$ defines a left action of $\ihom{S,S}$ on $S$.
\end{prp}
\begin{proof}
It follows directly from \autoref{prp:genProp:comp_counit} that the action of $\ihom{S,S}$ on $S$ is associative and that $\id_S$ acts as a unit.
\end{proof}


\begin{pgr}
\label{pgr:ring:inducedLtAction}
Let $R$ be a \CuSrg{}, let $S$ be a \CuSgp{}, and let $T$ be a \CuSgp{} with a left $R$-action $\alpha\colon R\otimes T\to T$.
Consider the general left unit map $R\otimes\ihom{S,T}\to\ihom{S,R\otimes T}$ from \autoref{dfn:genProp:genUnit}.
Postcomposing with $\alpha_\ast\colon\ihom{S,R\otimes T}\to\ihom{S,T}$ we obtain a \CuMor{} that we denote by $\alpha_S$:
\[
\alpha_S\colon R\otimes\ihom{S,T} \to \ihom{S, R\otimes T} \xrightarrow{\alpha_*} \ihom{S,T}.
\]

Let $r\in R$ and $x\in\ihom{S,T}$.
We denote $\alpha_S(r\otimes x)$ by $rx$.
Choose a path $\pathCu{f}=(f_\lambda)_\lambda$ in $\CatCuMor[S,T]$ representing $x$, and pick a path $(r_\lambda)_\lambda$ in $R$ with endpoint $r$.
For each $\lambda$, let $r_\lambda f_\lambda\colon S\to T$ be given by $s\mapsto r_\lambda f_\lambda(s)$.
Then $(r_\lambda f_\lambda)_\lambda$ is a path in $\CatCuMor[S,T]$ and
\[
rx
= [(r_\lambda f_\lambda)_\lambda].
\]
\end{pgr}

\begin{prp}
\label{prp:ring:inducedLtAction}
Let $R$ be a \CuSrg{} with compact unit, let $S$ be a \CuSgp{}, and let $T$ be a \CuSgp{} with a left $R$-action $\alpha\colon R\otimes T\to T$.
Then the map $\alpha_S\colon R\otimes\ihom{S,T}\to\ihom{S,T}$ from \autoref{pgr:ring:inducedLtAction} defines a left $R$-action on $\ihom{S,T}$.
\end{prp}
\begin{proof}
Let $r,r'\in R$ and $x\in\ihom{S,T}$.
Choose a path $\pathCu{f}=(f_\lambda)_\lambda$ in $\CatCuMor[S,T]$ representing $x$.
Choose paths $(r_\lambda)_\lambda$ and $(r'_\lambda)_\lambda$ in $R$ with endpoints $r$ and $r'$, respectively.
Then $(r_\lambda r'_\lambda)_\lambda$ is a path in $R$ with endpoint $rr'$.
Using the description of the $R$-action on $\ihom{S,T}$ from the end of \autoref{pgr:ring:inducedLtAction}, we deduce that
\[
(rr')x
= [((r_\lambda r'_\lambda) f_\lambda)_\lambda]
= [(r_\lambda (r'_\lambda f_\lambda))_\lambda]
= r[(r'_\lambda f_\lambda)_\lambda]
= r(r'x).
\]

Let $1$ denote the unit element of $R$.
For every $f\in\CatCuMor[S,T]$, we have $1f=f$.
Since $1$ is compact, the constant function with value $1$ is a path in $R$ with endpoint $1$.
It follows that
\[
1x
= [(1 f_\lambda)_\lambda]
= [(f_\lambda)_\lambda]
= x. \qedhere
\]
\end{proof}

\begin{prp}
\label{prp:ring:ihomST}
Let $S$ and $T$ be \CuSgp{s}.
Then the composition product $\circ\colon\ihom{T,T}\otimes\ihom{S,T}\to\ihom{S,T}$ defines a left action of the \CuSrg{} $\ihom{T,T}$ on $\ihom{S,T}$.
Analogously, we obtain a right action of $\ihom{S,S}$ on $\ihom{S,T}$.
These actions are compatible and thus $\ihom{S,T}$ is a $(\ihom{T,T},\ihom{S,S})$-$\CatCu$-semibimodule.
\end{prp}
\begin{proof}
This follows directly from the associativity of the composition product;
see \autoref{prp:genProp:compAssoc}.
\end{proof}

\begin{rmk}
\label{rmk:ring:ihomST}
Let $S$ and $T$ be \CuSgp{s}.
By \autoref{prp:ring:ihomSS_act_S}, the evaluation map $\counit_{T,T}\colon\ihom{T,T}\otimes T\to T$ from \autoref{dfn:genProp:counitMap} defines a left action of $\ihom{T,T}$ on $T$.
By \autoref{prp:ring:inducedLtAction}, this induces a left action of $\ihom{T,T}$ on $\ihom{S,T}$.
This action agrees with that from \autoref{prp:ring:ihomST}.
\end{rmk}

\begin{exa}
\label{exa:ring:ihomSS_noncomm}
Given $k\in\NN$, we let $\NNbar^k$ denote the Cuntz semigroup of the \ca{} $\CC^k$.
We think of $\pathCu{v}\in\NNbar^k$ as a tuple $(v_1,\ldots,v_k)^T$ with $k$ entries in $\NNbar$.
We let $\pathCu{e^{(1)}},\ldots,\pathCu{e^{(k)}}$ denote the `standard basis vectors' of $\NNbar^k$, such that $\pathCu{v}=\sum_{i=1}^k v_i \pathCu{e^{(i)}}$.

Let $k,l\in\NN$.
We claim that $\ihom{\NNbar^k,\NNbar^l}$ can be identified with $M_{l,k}(\NNbar)$, the $l\times k$-matrices with entries in $\NNbar$, with order and addition defined entrywise.
Thus, as a \CuSgp{}, $\ihom{\NNbar^k,\NNbar^l}$ is isomorphic to $\NNbar^{kl}$.
However, the presentation as matrices allows to expatiate the composition product.

First, let $\varphi\colon\NNbar^k\to\NNbar^l$ be a generalized \CuMor.
For each $j\in\{1,\ldots,k\}$, we consider the vector $\varphi(\pathCu{e^{(j)}})$ in $\NNbar^l$ and we let $x_{1,j},\ldots,x_{l,j}$ denote its coefficients.
This defines a matrix $\pathCu{x}=(x_{i,j})_{i,j}$ with $l\times k$ entries in $\NNbar$. It is then readily verified that
the coefficients of $\varphi(\pathCu{v})$ are obtained by multiplication of the matrix $\pathCu{x}$ with the vector of coefficients of $\pathCu{v}$.
We identify $\varphi$ with the associated matrix $\pathCu{x}$ in $M_{l,k}(\NNbar)$.

Let $\varphi,\psi\colon\NNbar^k\to\NNbar^l$ be generalized \CuMor{s} with associated matrices $\pathCu{x}$ and $\pathCu{y}$ in $M_{l,k}(\NNbar)$.
It is straightforward to check that $\varphi\prec\psi$ if and only if $x_{i,j}$ is finite and $x_{i,j}\leq y_{i,j}$ for each $i,j$, and thus the claim follows.

Given $k,l,m\in\NN$, consider the composition product
\[
\ihom{\NNbar^l,\NNbar^m} \otimes \ihom{\NNbar^k,\NNbar^l} \to \ihom{\NNbar^k,\NNbar^m}.
\]
After identifying $\ihom{\NNbar^k,\NNbar^l}$ with $M_{l,k}(\NNbar)$, identifying $\ihom{\NNbar^l,\NNbar^m}$ with $M_{m,l}(\NNbar)$, and identifying $\ihom{\NNbar^k,\NNbar^m}$ with $M_{m,k}(\NNbar)$, the composition product is given as a map
\[
M_{m,l}(\NNbar) \otimes M_{l,k}(\NNbar) \to M_{m,k}(\NNbar).
\]
It is straightforward to check that this map is induced by matrix multiplication.
In particular, the \CuSrg{} $\ihom{\NNbar^k,\NNbar^k}$ can be identified with the matrix ring $M_{k,k}(\NNbar)$.
Thus, for $k\geq 2$, the \CuSrg{} $\ihom{\NNbar^k,\NNbar^k}$ is not commutative.
\end{exa}

Given a \CuSrg{} $R$, recall that $\mu_R\colon R\otimes R\to R$ denotes the \CuMor{} induced by multiplication.

\begin{dfn}
\label{dfn:ring:pi}
Given a \CuSrg{} $R$, we let $\pi_R\colon R\to \ihom{R,R}$ be the \CuMor{} that corresponds to $\mu_R$ under the identification
\[
\CatCuMor\big( R, \ihom{R,R} \big)
\cong \CatCuMor\big( R\otimes R, R \big).
\]
\end{dfn}

\begin{rmk}
The \CuMor{} $\pi_R$ can be regarded as a kind of left regular representation of $R$.
\end{rmk}

\begin{lma}
\label{prp:ring:piUnit}
We have $\pi_R=(\mu_R)_*\circ\unit_{R,R}$ and $\counit_{R,R}\circ(\pi_R\otimes\id_R)=\mu_R$.
\end{lma}
\begin{proof}
The first equality follows from \autoref{prp:genProp:correspondence}.
It is straightforward to show that $\counit_{R,R}\circ((\mu_R)_*\otimes\id_R)=\mu_R\circ\counit_{R,R\otimes R}$.
Further, we have $\counit_{R,R\otimes R}\circ(\unit_{R,R}\otimes\id_R)=\id_{R\otimes R}$ by \autoref{prp:genProp:unit_counit}.
Using these equations, we deduce that
\begin{align*}
\counit_{R,R}\circ(\pi_R\otimes\id_R)
&= \counit_{R,R}\circ(((\mu_R)_*\circ\unit_{R,R})\otimes\id_R) \\
&= \counit_{R,R}\circ((\mu_R)_*\otimes\id_R)\circ(\unit_{R,R}\otimes\id_R) \\
&= \mu_R\circ\counit_{R,R\otimes R}\circ(\unit_{R,R}\otimes\id_R)
= \mu_R. \qedhere
\end{align*}
\end{proof}

\begin{thm}
\label{prp:ring:piMultiplicative}
Let $R$ be a \CuSrg.
Then $\pi_R\colon R\to\ihom{R,R}$ is multiplicative.
If the unit element of $R$ is compact, then $\pi_R$ is unital.
\end{thm}
\begin{proof}
Let $M\colon\ihom{R,R}\otimes\ihom{R,R}\to\ihom{R,R}$ denote the composition map.
We need to show that $M\circ(\pi_R\otimes\pi_R) = \pi_R\circ\mu_R$.

Given $r,s\in R$, choose paths $\pathCu{r}=(r_\lambda)_\lambda$ and $\pathCu{s}=(s_\lambda)$ in $(R,\ll)$ with endpoints $r$ and $s$, respectively.
For each $\lambda$, let $f_\lambda\colon R\to R$ and $g_\lambda\colon R\to R$ be the generalized \CuMor{} given by left multiplication with $r_\lambda$ and $s_\lambda$, respectively.
By \autoref{prp:genProp:unitMap}, we have $\unit_{R,R}(r) = [(r_\lambda\otimes\freeVar)_\lambda]$, where $r_\lambda\otimes\freeVar\colon R\to R\otimes R$ is the map sending $t\in R$ to $r_\lambda\otimes t$.
We also have $\mu_R\circ(r_\lambda\otimes\freeVar)=f_\lambda$.
Since $\pi_R=(\mu_R)_*\circ\unit_{R,R}$ by \autoref{prp:ring:piUnit}, it follows that $\pi_R(r)=[(f_\lambda)_\lambda]$.
Likewise, we deduce $\pi_R(s)=[(g_\lambda)_\lambda]$.
By \autoref{prp:genProp:composition}, we obtain $M(\pi_R(r)\otimes\pi_R(s))=[(f_\lambda\circ g_\lambda)_\lambda]$.

As the product in $R$ is associative, the composition $f_\lambda\circ g_\lambda$ is the generalized \CuMor{} $h_\lambda$ defined by left multiplication with $r_\lambda s_\lambda$.
Notice that $(r_\lambda s_\lambda)_\lambda$ is a path in $(R,\ll)$ with endpoint $rs$.
Therefore, $\pi_R(rs)=[(h_\lambda)_\lambda]$.
Altogether, we get the desired equality
\[
M(\pi_R(r)\otimes\pi_R(s))=[(f_\lambda\circ g_\lambda)_\lambda]=[(h_\lambda)_\lambda]=\pi_R(\mu_R(r\otimes s)).
\]

To show the second statement, let us assume that the unit $1_R$ of $R$ is compact.
Then the constant function with value $1_R$ is a path in $(R,\ll)$ with endpoint $1_R$. It follows easily as in the first part of the proof that $\pi_R(1_R)=[(\id_R)_\lambda]=\id_R$.
\end{proof}

\begin{dfn}
\label{dfn:ring:eps}
Let $R$ be a \CuSrg{} with unit $1_R$.
Then $\varepsilon_R\colon \ihom{R,R}\to R$ is the generalized \CuMor{} given by
\[
\varepsilon_R([\pathCu{f}]) = \sup_\lambda f_\lambda(1_R),
\]
for a path $\pathCu{f}=(f_\lambda)_\lambda$ in $\CatCuMor[R,R]$.
\end{dfn}

\begin{rmk}
\label{rmk:ring:eps}
Let $\sigma_{R,R}\colon\ihom{R,R}\to\CatCuMor[R,R]$ denote the endpoint map as in \autoref{dfn:endpoint}, 
Then $\varepsilon_R(x)=\sigma_{R,R}(x)(1_R)$ for every $x\in\ihom{R,R}$.
\end{rmk}

\begin{lma}
\label{prp:ring:eps}
We have $\varepsilon_R\circ\pi_R = \id_R$.
\end{lma}
\begin{proof}
Given $r\in R$, choose a path $(r_\lambda)_\lambda$ in $(R,\ll)$ with endpoint $r$, and for each $\lambda$ let $f_\lambda\colon R\to R$ be given by left multiplication with $r_\lambda$.
As in the proof of \autoref{prp:ring:piMultiplicative}, we obtain $\pi_R(r)=[(f_\lambda)_\lambda]$, whence
\[
\varepsilon_R(\pi_R(r))
= \varepsilon_R([(f_\lambda)_\lambda])
= \sup_\lambda f_\lambda(1_R)
= \sup_\lambda (r_\lambda 1_R)
= r. \qedhere
\]
\end{proof}

\begin{prp}
\label{prp:ring:R_sub_ihomRR}
Let $R$ be a \CuSrg.
Then $\pi_R\colon R\to\ihom{R,R}$ is a multiplicative order-embedding.
Thus, in a natural way, $R$ is a sub-semiring of $\ihom{R,R}$.
If the unit of $R$ is compact, then $R$ is even a unital sub-semiring of $\ihom{R,R}$.
\end{prp}
\begin{proof}
By \autoref{prp:ring:eps}, we have $\varepsilon_R\circ\pi_R = \id_R$, which implies that $\pi_R$ is an order-embedding.
By \autoref{prp:ring:piMultiplicative}, $\pi_R$ is a (unital) multiplicative \CuMor.
\end{proof}

An element $a$ in a \CuSgp{} $S$ is \emph{soft} if for every $a'\in S$ with $a'\ll a$ there exists $k\in\NN$ with $(k+1)a'\leq ka$;
see \cite[Definition~5.3.1]{AntPerThi18:TensorProdCu}. 
The following result will be used below.

\begin{lma}
\label{prp:appl:genCuPresSoft}
Let $S$ and $T$ be \CuSgp{s}, let $\varphi\colon S\to T$ be a generalized \CuMor{}, and let $a\in S$ be soft.
Then $\varphi(a)$ is soft.
\end{lma}
\begin{proof}
To verify that $\varphi(a)$ is soft, let $x\in T$ satisfy $x\ll\varphi(a)$.
Using that $\varphi$ preserves suprema of increasing sequences, we can choose $a'\in S$ with $a'\ll a$ and $x\leq\varphi(a')$.
(Indeed, applying \axiomO{2} in $S$, choose a $\ll$-increasing sequence $(a_n)_n$ in $S$ with supremum $a$.
Then $\varphi(a)=\sup_n\varphi(a_n)$, whence there is $n$ with $x\leq\varphi(a_n)$.)
Since $a$ is soft, we can choose $k\in\NN$ such that $(k+1)a'\leq ka$.
Then
\[
(k+1)x
\leq (k+1) \varphi(a')
= \varphi((k+1)a')
\leq \varphi(ka)
= k\varphi(a). \qedhere
\]
\end{proof}

Let $\PPbar=[0,\infty]$, with natural order and addition. Recall that $\PPbar$ is isomorphic to the Cuntz semigroup of the Jacelon-Razak algebra (see \cite{Jac13Projectionless} and \cite{Rob13Cone}). The usual multiplication of real numbers extends to $\PPbar$.
This gives $\PPbar$ the structure of a commutative \CuSrg.

\begin{exa}
\label{exa:ring:M1}
Let $M_1=[0,\infty)\sqcup (0,\infty]$, a disjoint union of compact elements $[0,\infty)$ and nonzero soft elements $(0,\infty]$. Recall that $M_1$ denotes the Cuntz semigroup of a $\mathrm{II}_1$-factor;
see \cite[Example 4.14]{AntPerThi17arX:AbsBivariantCu} and \cite[Proposition 4.15]{AntPerThi17arX:AbsBivariantCu}.
We identify $\PPbar=[0,\infty]$ with the sub-\CuSgp{} of soft elements in $M_1$, and we define the \CuMor{} $\varrho\colon M_1\to\PPbar\subseteq M_1$ by fixing all soft elements and by sending a compact to the soft element of the same value.

We define a product on $M_1$ as follows:
We equip the compact part $[0,\infty)$ with the usual multiplication of real numbers, and similarly for the product in $(0,\infty]$.
The product of any element with $0$ is $0$.
Given a nonzero compact element $a$ and a nonzero soft element $b$, their product is defined as the soft element $ab:=\varrho(a)b$.

This gives $M_1$ the structure of a commutative $\CatCu$-semiring.
Moreover, we may identify $\PPbar$ with the (nonunital) sub-$\CatCu$-semiring of soft elements in $M_1$.
The map $\varrho\colon M_1\to\PPbar$ is multiplicative.
One can show that the map $\pi_{M_1}\colon M_1\to\ihom{M_1,M_1}$ is an isomorphism.
\end{exa}

\begin{exa}
\label{exa:ring:PPbar}
We have $\ihom{\PPbar,\PPbar}\cong M_1$.
The map $\pi_{\PPbar}\colon \PPbar\to\ihom{\PPbar,\PPbar}$ embeds $\PPbar$ as the sub-$\CatCu$-semiring of soft elements in $M_1$.
In particular, $\pi_{\PPbar}$ is not unital.
\end{exa}
\begin{proof}
We have $\ihom{\PPbar,\PPbar}\cong M_1$ by \cite[Proposition 5.13]{AntPerThi17arX:AbsBivariantCu}.
By \autoref{prp:ring:R_sub_ihomRR}, $\pi_{\PPbar}$ is a multiplicative order-embedding.
Note that every element of $\PPbar$ is soft.
By \autoref{prp:appl:genCuPresSoft}, a generalized \CuMor{} maps soft elements to soft elements.
Thus, the image of $\pi_{\PPbar}$ is contained in the soft elements of $M_1$.
It easily follows that $\pi_{\PPbar}$ identifies $\PPbar$ with the soft elements in $M_1$.
Since the unit of $M_1$ is compact, it also follows that $\pi_{\PPbar}$ is not unital.
\end{proof}

We finally turn our attention to solid semirings.
Recall from \cite[Definition~7.1.5]{AntPerThi18:TensorProdCu} that a \CuSrg{} $R$ is said to be \emph{solid} if $\mu_R\colon R\otimes R\to R$ is an isomorphism.
In \cite{AntPerThi18:TensorProdCu}, all \CuSrg{s} were required to be commutative, and thus a solid \CuSrg{} was assumed to be commutative.
Next, we show that this assumption is not necessary since a \CuSrg{} is automatically commutative as soon as $\mu_R$ is injective.

\begin{lma}
\label{prp:ring:mu_inj}
Let $R$ be a \CuSrg{} such that $\mu_R\colon R\otimes R\to R$ is injective.
Then $R$ is commutative and $\mu_R$ is an isomorphism (and consequently $R$ is solid.)
\end{lma}
\begin{proof}
To show that $R$ is commutative, let $a,b\in R$.
We have
\[
\mu_R(1\otimes a)=a=\mu_R(a\otimes 1),
\]
and thus $1\otimes a=a\otimes 1$ in $R\otimes R$.
Consider the shuffle \CuMor{} $\alpha\colon R\otimes R\otimes R\to R\otimes R\otimes R$ that satisfies $\alpha(x\otimes y\otimes z)=y\otimes x\otimes z$ for every $x,y,z\in R$.
Then
\[
1\otimes b \otimes a
= \alpha(b\otimes 1\otimes a)
= \alpha(b\otimes a\otimes 1)
= a\otimes b\otimes 1
\]
in $R\otimes R\otimes R$.
By the associativity of the product in $R$, we get $ba=ab$, as desired.

Thus, $R$ is commutative and $1\otimes a=a\otimes 1$ in $R\otimes R$, for every $a\in R$.
Using \cite[Proposition~7.1.6]{AntPerThi18:TensorProdCu}, this implies that $R$ is solid.
\end{proof}

Let $R$ be a solid \CuSrg, and let $S$ be a \CuSgp.
It was shown in \cite[Corollary~7.1.8]{AntPerThi18:TensorProdCu} that any two $R$-actions on $S$ agree.
(Since $R$ is commutative, we need not distinguish between left and right $R$-actions.)
Thus, $S$ either has a (unique) $R$-action, or it does not admit any $R$-action, which means that having an $R$-action is a property rather than an additional structure for $S$, which justifies the following definition.

\begin{dfn}
\label{dfn:ring:Rstable}
Let $R$ be a solid \CuSrg, and let $S$ be a \CuSgp.
We say that $S$ is \emph{$R$-stable} if $S$ has an $R$-action.
\end{dfn}

\begin{rmk}
\label{rmk:ring:Rstable}
In \cite{AntPerThi18:TensorProdCu}, we said that $S$ `has $R$-multiplication' if it has an $R$-action.
Given a solid ring $R$, it was shown \cite[Theorem~7.1.12]{AntPerThi18:TensorProdCu} that $S$ is $R$-stable if and only if $S\cong R\otimes S$.

Recall that a \ca{} $A$ is said to be \emph{$\mathcal{Z}$-stable} if $A\cong\mathcal{Z}\otimes A$, and similarly one defines being UHF-stable and $\mathcal{O}_\infty$-stable.
Thus, the terminology of being `$R$-stable' for \CuSgp{s} is analogous to the terminology used for \ca{s}.
\end{rmk}

\begin{thm}
\label{prp:ring:charSolid}
Given a \CuSrg{} $R$, consider the following statements:
\begin{enumerate}
\item
$R$ is solid, that is, $\mu\colon R\otimes R\to R$ is an isomorphism.
\item
The map $\counit_{R,R}\colon\ihom{R,R}\otimes R\to R$ is an isomorphism.
\item
The map $\pi_R\otimes\id_R\colon R\otimes R\to \ihom{R,R}\otimes R$ is an isomorphism.
\item
The map $\pi_R\colon R\to \ihom{R,R}$ is an isomorphism.
\item
The map $\varepsilon_R\colon \ihom{R,R}\to R$ is an isomorphism.
\end{enumerate}
Then the following implications hold:
\[
(1) \Leftarrow (2) \Rightarrow (3) \Leftarrow (4) \Leftrightarrow (5).
\]
Further, if $R$ satisfies (1) and (3), then it satisfies~(2).
The \CuSrg{} $\PPbar$ satisfies~(1),(2) and~(3), but not~(4);
see \autoref{exa:ring:PPbar}.
The \CuSrg{} $M_1$ satisfies~(3) and~(4) but neither~(1) nor~(2);
see \autoref{exa:ring:M1}.
\end{thm}
\begin{proof}
By \autoref{prp:ring:eps}, we have $\varepsilon_R\circ\pi_R = \id_R$.
It follows that $\varepsilon_R$ is an isomorphism if and only if $\pi_R$ is, which shows the equivalence of~(4) and~(5).
It is obvious that~(4) implies~(3).
To show that~(2) implies~(1), assume that $\counit_{R,R}$ is an isomorphism.
We have
\[
(\varepsilon_R\otimes\id_R)\circ(\pi_R\otimes\id_R) = \id_R\otimes\id_R,
\]
which shows that $\pi_R\otimes\id_R$ is an order-embedding.
Hence, $\counit_{R,R}\circ(\pi_R\otimes\id_R)$ is an order-embedding.
By \autoref{prp:ring:piUnit}, we have $\counit_{R,R}\circ(\pi_R\otimes\id_R)=\mu_R$, whence $\mu_R$ is an order-embedding.
By \autoref{prp:ring:mu_inj}, this implies that $R$ is solid.

Using again that $\counit_{R,R}\circ(\pi_R\otimes\id_R)=\mu_R$, if any two of the three maps $\counit_{R,R}$, $\pi_R\otimes\id_R$ and $\mu_R$ are isomorphisms, then so is the third.
This shows that~(2) implies~(3), and that the combination of~(1) and~(3) implies~(2).
\end{proof}

\begin{qst}
\label{qst:ring:charSolid}
Given a solid \CuSrg{} $R$, is the evaluation map $\counit_{R,R}\colon\ihom{R,R}\otimes R\to R$ an isomorphism?
\end{qst}

\begin{rmk}
\label{rmk:ring:charSolid}
Let $R$ be a solid \CuSrg{}.
The answer to \autoref{qst:ring:charSolid} is `yes' in the following cases:
\begin{enumerate}
\item
If the unit of $R$ is compact;
see \autoref{rmk:ring:ihomRSisS} below.
\item
If $R$ satisfies \axiomO{5} and \axiomO{6}.
This follows from the classification of solid \CuSrg{s} with \axiomO{5} obtained in \cite[Theorem~8.3.13]{AntPerThi18:TensorProdCu} which shows that each such \CuSrg{} is either isomorphic to $\PPbar$ or has a compact unit.
In either case, \autoref{qst:ring:charSolid} has a positive answer.
\end{enumerate}

In particular, a \CuSrg{} $R$ with compact unit is solid if and only if the evaluation map $\counit_{R,R}\colon\ihom{R,R}\otimes R\to R$ is an isomorphism.
\end{rmk}

\begin{thm}
\label{thm:Qmultiplication}
Let $R$ be a solid \CuSrg{} with compact unit, and let $S$ and $T$ be \CuSgp{s}.
Assume that $T$ is $R$-stable.
Then $\ihom{S,T}$ is $R$-stable, and hence $\ihom{S,T}\cong R\otimes\ihom{S,T}$.
\end{thm}
\begin{proof}
Since the unit of $R$ is compact, it follows from \autoref{prp:ring:inducedLtAction} that $\ihom{S,T}$ has a left $R$-action.
Since $R$ is solid, this implies that $\ihom{S,T}$ is $R$-stable.
\end{proof}

\begin{lma}
\label{prp:ring:compRStoT}
Let $R$ be a solid \CuSrg{}, let $S$ and $T$ be \CuSgp{s}, and let $f,g\colon R\otimes S\to T$ be a generalized \CuMor{s}.
Assume that $T$ is $R$-stable.
Then $f\leq g$ if and only if $f(1\otimes a)\leq g(1\otimes a)$ for all $a\in S$.

If the unit of $R$ is compact, then $f\prec g$ if and only if $f(1\otimes a')\ll g(1\otimes a)$ for all $a',a\in S$ with $a'\ll a$.
\end{lma}
\begin{proof}
The forward implications are obvious.
To show the converse of the first statement, assume that $f(1\otimes a)\leq g(1\otimes a)$ for all $a\in S$.
To verify $f\leq g$, it is enough to show that $f(r\otimes a)\leq g(r\otimes a)$ for all $r\in R$ and $a\in S$.
Note that $R\otimes S$ and $T$ are $R$-stable.
Since $R$ is solid, every generalized $\CatCu$-morphism between $R$-stable $\CatCu$-semigroups is automatically $R$-linear; see \cite[Proposition~7.1.6]{AntPerThi18:TensorProdCu}.
Thus, given $r\in R$ and $a\in S$, we obtain
\[
f(r\otimes a)
= f(r(1\otimes a))
= r f(1\otimes a)
\leq r g(1\otimes a)
= g(r\otimes a).
\]

To show the converse of the second statement, assume that $f(1\otimes a')\ll g(1\otimes a)$ for all $a',a\in S$ with $a'\ll a$.
To verify $f\prec g$, it is enough to show that $f(r'\otimes a')\ll g(r\otimes a)$ for all $r',r\in R$ and $a',a\in S$ with $r'\ll r$ and $a'\ll a$.
Given such $r',r,a'$ and $a$, we use at the second step that multiplication in $R$ preserves the joint way-below relation, to deduce
\[
f(r'\otimes a')
= r' f(1\otimes a')
\ll r g(1\otimes a)
= g(r\otimes a). \qedhere
\]
\end{proof}

\begin{prp}
\label{prp:ring:ihomRSTisST}
Let $R$ be a solid \CuSrg{} with compact unit, and let $S$ and $T$ be \CuSgp{s}.
Assume that $T$ is $R$-stable.
Let $\alpha\colon S\to R\otimes S$ be the \CuMor{} given by $\alpha(a)=1\otimes a$, for $a\in S$.
Then the induced map $\alpha^*\colon\ihom{R\otimes S,T}\to\ihom{S,T}$ is an isomorphism.
\end{prp}
\begin{proof}
Consider $\alpha_\CatQ^*\colon\CatCuMor[R\otimes S,T]\to\CatCuMor[S,T]$ given by sending a generalized \CuMor{} $f\colon R\otimes S\to T$ to the generalized \CuMor{} $\alpha_\CatQ^*(f)$ given by
\[
\alpha_\CatQ^*(f)(a) = f(1\otimes a),
\]
for $a\in S$.
It follows from \autoref{prp:ring:compRStoT} that $\alpha_\CatQ^*$ is an isomorphism of $\CatQ$-semigroups.
Since $\alpha^*$ is obtained by applying the functor $\tau$ to $\alpha_Q^*$ (see the comments after \cite[Remark~5.4]{AntPerThi17arX:AbsBivariantCu}), it follows that $\alpha^*$ is an isomorphism, as desired.
\end{proof}

\begin{cor}
\label{prp:ring:ihomRSisS}
Let $R$ be a solid \CuSrg{} with compact unit, and let $T$ be an $R$-stable \CuSgp{}.
Then there is a natural isomorphism $\ihom{R,T}\cong T$.
\end{cor}
\begin{proof}
Applying \autoref{prp:ring:ihomRSTisST} for $S:=\NNbar$, we obtain $\ihom{R,T}\cong\ihom{\NNbar,T}$.
By \autoref{prp:genProp:iIso}, we have a natural isomorphism $\ihom{\NNbar,T}\cong T$.
\end{proof}

\begin{rmk}
\label{rmk:ring:ihomRSisS}
Let $R$ be a solid \CuSrg{} with compact unit.
Since $R$ is $R$-stable itself, it follows from \autoref{prp:ring:ihomRSisS} that $\ihom{R,R}\cong R$.
It follows that the evaluation map $\counit_{R,R}\colon\ihom{R,R}\otimes R\to R$ is an isomorphism.

For the solid \CuSrg{} $\PPbar$, we have seen in \cite[Proposition 5.13]{AntPerThi17arX:AbsBivariantCu} that $\ihom{\PPbar,\PPbar}\cong M_1 \ncong \PPbar$.
This shows that \autoref{prp:ring:ihomRSTisST} and \autoref{prp:ring:ihomRSisS} cannot be generalized to solid \CuSrg{s} without compact unit.
\end{rmk}




\providecommand{\etalchar}[1]{$^{#1}$}
\providecommand{\bysame}{\leavevmode\hbox to3em{\hrulefill}\thinspace}
\providecommand{\noopsort}[1]{}
\providecommand{\mr}[1]{\href{http://www.ams.org/mathscinet-getitem?mr=#1}{MR~#1}}
\providecommand{\zbl}[1]{\href{http://www.zentralblatt-math.org/zmath/en/search/?q=an:#1}{Zbl~#1}}
\providecommand{\jfm}[1]{\href{http://www.emis.de/cgi-bin/JFM-item?#1}{JFM~#1}}
\providecommand{\arxiv}[1]{\href{http://www.arxiv.org/abs/#1}{arXiv~#1}}
\providecommand{\doi}[1]{\url{http://dx.doi.org/#1}}
\providecommand{\MR}{\relax\ifhmode\unskip\space\fi MR }
\providecommand{\MRhref}[2]{%
	\href{http://www.ams.org/mathscinet-getitem?mr=#1}{#2}
}
\providecommand{\href}[2]{#2}

\end{document}